\documentclass[11pt]{amsart}

\usepackage{amssymb}
\usepackage{url}
\usepackage{bigints}
\usepackage{mathrsfs}
\usepackage{tikz-cd}
\usetikzlibrary{decorations.pathmorphing}
	\normalbaselines						  %

\def\RR{\mathbb{R}}

\def\CC{\mathbb{C}}
\def\NN{\mathbb{N}}
\newcommand{\ID}{{\mathbf{1}}}
\newcommand{\al}{{\alpha}}
\newcommand{\la}{{\lambda}}
\newcommand{\f}{{\varphi}}
\newcommand{\IM}{{\operatorname{Im}}}

\newcommand{\Pro}{{\mathbb{P}}}

\newcommand{\R}{{\mathbb  R}}

\newcommand{\Z}{{\mathbb  Z}}

\newcommand{\N}{{\mathbb  N}}
\newcommand{\C}{{\mathbb  C}}

\newcommand{\OID}{{\mathbf{I}}}

\newcommand{\fdot}{\,\cdot\,}

\makeatletter
\def\Ddots{\mathinner{\mkern1mu\raise\p@
\vbox{\kern7\p@\hbox{.}}\mkern2mu
\raise4\p@\hbox{.}\mkern2mu\raise7\p@\hbox{.}\mkern1mu}}
\makeatother

\newcommand{\cH}{\mathcal{H}}

\newcommand{\p}{\mathbb{P}}

\DeclareMathOperator{\spa}{span}

\DeclareMathOperator{\Mod}{mod}
\DeclareMathOperator{\supp}{supp}

\newcommand{\ci}[1]{_{ {}_{\scriptstyle #1}}}
\newcommand{\ti}[1]{_{\scriptstyle \text{\rm #1}}}
\catcode`@=11
\chardef\mathlig@atcode\count255

\def\actively#1#2{\begingroup\uccode`\~=`#2\relax\uppercase{\endgroup#1~}}
\def\mathlig@gobble{\afterassignment\mathlig@next@cmd\let\mathlig@next= }
\def\mathlig@delim{\mathlig@delim}
\def\mathlig@defcs#1{\expandafter\def\csname#1\endcsname}
\def\mathlig@let@cs#1#2{\expandafter\let\expandafter#1\csname#2\endcsname}
\def\mathlig@appendcs#1#2{\expandafter\edef\csname#1\endcsname{\csname#1\endcsname#2}}

\def\mathlig#1#2{\mathlig@checklig#1\mathlig@end\mathlig@defcs{mathlig@back@#1}{#2}\ignorespaces}

\def\mathlig@checklig#1#2\mathlig@end{%
 \expandafter\ifx\csname mathlig@forw@#1\endcsname\relax
 \expandafter\mathchardef\csname mathlig@back@#1\endcsname=\mathcode`#1%
 \mathcode`#1"8000\actively\def#1{\csname mathlig@look@#1\endcsname}%
 \mathlig@dolig#1\mathlig@delim
\fi
\mathlig@checksuffix#1#2\mathlig@end
}

\def\mathlig@checksuffix#1#2\mathlig@end{%
\ifx\mathlig@delim#2\mathlig@delim\relax\else\mathlig@checksuffix@{#1}#2\mathlig@end\fi
}
\def\mathlig@checksuffix@#1#2#3\mathlig@end{%
\expandafter\ifx\csname mathlig@forw@#1#2\endcsname\relax\mathlig@dosuffix{#1}{#2}\fi
\mathlig@checksuffix{#1#2}#3\mathlig@end
}

\def\mathlig@dosuffix#1#2{%
\mathlig@appendcs{mathlig@toks@#1}{#2}%
\mathlig@dolig{#1}{#2}\mathlig@delim
}

\def\mathlig@dolig#1#2\mathlig@delim{%
 \mathlig@defcs{mathlig@look@#1#2}{%
 \mathlig@let@cs\mathlig@next{mathlig@forw@#1#2}\futurelet\mathlig@next@tok\mathlig@next}%
 \mathlig@defcs{mathlig@forw@#1#2}{%
  \mathlig@let@cs\mathlig@next{mathlig@back@#1#2}%
  \mathlig@let@cs\checker{mathlig@chck@#1#2}%
  \mathlig@let@cs\mathligtoks{mathlig@toks@#1#2}%
  \expandafter\ifx\expandafter\mathlig@delim\mathligtoks\mathlig@delim\relax\else
  \expandafter\checker\mathligtoks\mathlig@delim\fi
  \mathlig@next
 }%
 \mathlig@defcs{mathlig@toks@#1#2}{}%
 \mathlig@defcs{mathlig@chck@#1#2}##1##2\mathlig@delim{%
  \ifx\mathlig@next@tok##1%
   \mathlig@let@cs\mathlig@next@cmd{mathlig@look@#1#2##1}\let\mathlig@next\mathlig@gobble
  \fi
  \ifx\mathlig@delim##2\mathlig@delim\relax\else
   \csname mathlig@chck@#1#2\endcsname##2\mathlig@delim
  \fi
 }%
 \ifx\mathlig@delim#2\mathlig@delim\else
  \mathlig@defcs{mathlig@back@#1#2}{\csname mathlig@back@#1\endcsname #2}%
 \fi
}%

\catcode`@\mathlig@atcode

\mathchardef\ordinarycolon\mathcode`\:
\def\vcentcolon{\mathrel{\mathop\ordinarycolon}}
\mathlig{:=}{\vcentcolon=}
\mathlig{::=}{\vcentcolon\vcentcolon=}

\numberwithin{equation}{section}

\theoremstyle{plain}
\newtheorem{theo}{Theorem}[section]
\newtheorem{cor}[theo]{Corollary}
\newtheorem{lem}[theo]{Lemma}
\newtheorem{prop}[theo]{Proposition}

\theoremstyle{definition}

\newtheorem*{theorem*}{Theorem}
\theoremstyle{remark}

\theoremstyle{remark}
\newtheorem*{exs*}{Examples}
\theoremstyle{remark}
\newtheorem*{rems*}{Remarks}
\theoremstyle{remark}
\newtheorem*{rem*}{Remark}

\title[Iterated Rank-One Perturbations]{Spectral Analysis of Iterated Rank-One Perturbations}
\author{Dale~Frymark}
\address{Department of Mathematics, Stockholm University, Kr\"aftriket 6, 106 91 Stockholm, Sweden}
\email{dale@math.su.se}

\thanks{The work of Constanze Liaw was partially supported by the NSF grant DMS-1802682.}
\author{Constanze~Liaw}
\address{Department of Mathematical Sciences, University of Delware, 501 Ewing Hall, Newark, DE  19716, USA; and 
CASPER, Baylor University, One Bear Place \#97328,      
 Waco, TX  76798, USA.}
\email{liaw@udel.edu}

\keywords{Rank-One Perturbations, Random Schr\"odinger Operator, Spectral Theory, Perturbation Theory}
 \subjclass[2010]{81Q15, 47A55, 47B80, 82B44, 81Q10}

\begin{document}
%%%%%%%%%%%%%%%%%%%%%%%%%%%%%%%
%%%%%%%%%%%%%%%%%%%%%%%%%%%%%%%

\begin{abstract}
The authors study the spectral theory of self-adjoint operators that are subject to certain types of perturbations.

An iterative introduction of infinitely many randomly coupled rank-one perturbations is one of our settings. Spectral theoretic tools are developed to estimate the remaining absolutely continuous spectrum of the resulting random operators. Curious choices of the perturbation directions that depend on the previous realizations of the coupling parameters are assumed, and unitary intertwining operators are used. An application of our analysis shows localization of the random operator associated to the Rademacher potential.

Obtaining fundamental bounds on the types of spectrum under rank-one perturbation, without restriction on its direction, is another main objective. This is accomplished by analyzing Borel/Cauchy transforms centrally associated with rank-one perturbation problems. 
\end{abstract}

\maketitle

\setcounter{tocdepth}{1}
\tableofcontents

\setcounter{tocdepth}{2}

%%%%%%%%%%%%%%%%%%%%%%%%%%%%%%%
%%%%%%%%%%%%%%%%%%%%%%%%%%%%%%%
\section{Introduction}
%%%%%%%%%%%%%%%%%%%%%%%%%%%%%%%
%%%%%%%%%%%%%%%%%%%%%%%%%%%%%%%

An important branch of perturbation theory is the study of spectral properties of the sum $H+V$ of self-adjoint operators $H$ and $V$, under the assumption that the spectrum of $H$ is known and $V$ is from some operator class. The operator $V$ is often thought of as a potential. Our focus is on two classes for $V$, rank-one and certain probabilistic potentials.

Rank-one perturbations are considered a very simple type of perturbation as their range is one-dimensional. Let $T$ be a self-adjoint operator $T$ on a separable Hilbert space $\cH$, and consider the family of self-adjoint rank-one perturbations in the direction of a cyclic vector $\f\in\cH$:
\[
T_\alpha = T+\alpha \langle\fdot, \f\rangle\ci{\cH}\f,\qquad\alpha \in \R.
\]
For details concerning this definition, see equation \eqref{e-rk1} below.

More general rank-one perturbations can be defined when $H$ is unbounded, for instance through the theory of quadratic forms (see e.g.~\cite{LiawTreil1, SIMREV, Simon} and the references therein). Unbounded perturbatiuons are outside the scope of this paper, as our focus lays either on the iterative introduction of infinitely many rank-one perturbations, or on obtaining general bounds for single ones. 

A rank-one perturbation with $\f\in \cH$ is not only a compact operator. It is also Hilbert--Schmidt, trace class and even of finite rank (rank-one). Yet, their study has revealed an extremely subtle nature. For example, a description of the singular continuous spectrum of the perturbed operator $T_\alpha$ in terms of properties of the unperturbed operator $T$ is unknown, see e.g.~\cite{SIMREV} and the references within. Moreover, beyond the realms of mathematical physics and spectral analysis of self-adjoint operators, the problem of rank one perturbations is connected to many interesting topics in analysis, see e.g.~\cite{CimaMathesonRoss, t-KR, LiawTreil1, poltsara2006} and the references therein. The results in Section \ref{s-northo} contribute bounds of two types: a bound for how much absolutely continuous spectrum can be transferred to discrete spectrum, and a bound for how much mass from the discrete spectrum can be transferred to absolutely continuous spectrum via a rank-one perturbation.

As an object of great interest to mathematical physicists, Anderson-type Hamiltonians $H_\omega = H+V_\omega$ are a generalization of the discrete random Schr\"odinger operator with a probabilistic  potential $V_\omega  = \sum \omega_n (\fdot, \f_n)\ci\cH \f_n$. The perturbation problem is defined rigorously below in Subsection \ref{ss-01}. In this setup, the perturbation $V_\omega$ is almost surely a non-compact operator. As a result, none of classical perturbation theory applies.% From the perturbation theoretic perspective, rank-one perturbations and Anderson-type Hamiltonians are very different.

In 1958 P.W. Anderson (Subsection \ref{ss-01}, \cite{Anderson}) suggested that sufficiently large impurities in a semi-conductor could lead to spatial localization of electrons, called Anderson localization. The field has grown into a rich theory and is studied by both the physics and the mathematics community. There are many well-studied and famous open problems in this research area, one of which is the Anderson localization conjecture for the discrete random Schr\"odinger operator at weak disorder in two spacial dimensions \cite{Anderson, CFKS, Banff} (or delocalization conjecture \cite{Liaw2}). There are numerous ways to interpret the meaning of extended states throughout the literature. The current work in Section \ref{s-infinite} is related to the localization conjectures when the existence of \emph{extended states} is defined as almost surely non-trivial absolutely continuous spectrum in the Anderson-type Hamiltonian. This is sometimes referred to as spectral delocalization. Spectral localization thus refers to an Anderson-type Hamiltonian with trivial absolutely continuous part. 
 
Although rank-one perturbations and Anderson-type Hamiltonians are opposite in a perturbation theoretic sense, they have been found to be intimately connected \cite{AbukamovPoltoLiaw, JakLast2000, JakLast2006, KingKirbyLiaw, Liaw, Sim1994, SimonWolff}. Here we present further results linking these perturbation problems. Countably many rank-one perturbations are successively applied to a self-adjoint operator, $T$, with only absolutely continuous spectrum on a separable Hilbert space $\cH$. Specifically, we utilize the Aronszajn--Donoghue theory to determine the amount by which the absolutely continuous spectrum decays with each perturbation, and explicitly compute formulas describing how the initial spectrum changes after an infinite number of such perturbations. This construction involves a curious choice of the perturbation vector at each step in order to control properties of the perturbed operators in terms of the initial operator. In the limiting case, the infinitely perturbed operator is somewhat similar to an Anderson-type Hamiltonian and can be compared to the discrete random Schr\"odinger operator.

To avoid any possible confusion, we list differences between the construction in the current project and classical Anderson-type Hamiltonians:
\begin{enumerate}
\item The main distinction is that the iterative construction requires knowledge of the previous perturbation vectors, $\f_{n}$ for $n=1,\hdots, k-1$, as well as all \emph{previous} realizations of the (random) coupling parameter in order to choose the next perturbation vector, $\f_k$. This is very different from Anderson-type Hamiltonians, where the vectors $\f_k$'s form a sequence of orthonormal vectors and are given a priori, independently from the particular realization of the Anderson-type Hamiltonian. While our limiting operator is  similar to an Anderson-type Hamiltonian, it is mainly due to the specific choice of vectors $\f_k$ that it cannot be classified as such.

\item The construction yields an operator of spectral multiplicity one. The spectral multiplicity of a general Anderson-type Hamiltonian may not necessarily equal one. In fact, not even the spectral multiplicity of the discrete random Schr\"odinger operator is known, though it is suspected to be infinite (see e.g.~\cite{JakLast2000}).

\item We start on the spectral representation side of the operator. Hence, all of Lebesgue theory can be utilized as a tool, and rank-one perturbation theory becomes more concrete. This is not a serious distinction, as a unitary transformation takes any cyclic operator to its spectral representation.
\end{enumerate}

A primary development is the explicit calculation of the remaining absolutely continuous spectrum after an infinite number of rank-one perturbations. As suggested by Example \ref{example} below, precise control for even rank-two perturbations are challenging to produce and tend to be less explicit than those for rank-one perturbations. Recently, some progress was made for finite rank perturbations \cite{LiawTreil_arXiv} using matrix-valued measures. But the nature of our construction has a  different focus.

The probability measures that can be handled include the case of Rademacher potentials; see Subsection \ref{ss-rademacher}, where we provide results with Rademacher potential. These represent a worst-case scenario for our construction.

%%%%%%%%%%%%%%%%%%%%%%%%%%%%%%%
%%%%%%%%%%%%%%%%%%%%%%%%%%%%%%%
\subsection*{Acknowledgments.} The authors thank Alexei Poltoratski for suggesting some questions which led to this paper as well as for  many insightful discussions and comments.
%%%%%%%%%%%%%%%%%%%%%%%%%%%%%%%
%%%%%%%%%%%%%%%%%%%%%%%%%%%%%%%

%%%%%%%%%%%%%%%%%%%%%%%%%%%%%%%
%%%%%%%%%%%%%%%%%%%%%%%%%%%%%%%
\subsection{Outline}
%%%%%%%%%%%%%%%%%%%%%%%%%%%%%%%
%%%%%%%%%%%%%%%%%%%%%%%%%%%%%%%

The main tools of perturbation theory from Section \ref{ss-Pert} are utilized in Section \ref{s-firstpert}, where the majority of preparatory calculations take place, including applying Aronszajn--Donoghue theory to the first perturbation. Beginning with a measure that is constant over the interval $[-1,1]$, Aronszajn--Donoghue theory says that a perturbation creates a point mass outside of $[-1,1]$ and the remaining absolutely continuous spectrum is reduced accordingly. The precise strength of the point mass is calculated, and although it is possible to explicitly find a formula for the absolutely continuous spectral measure, it is avoided here for simplicity. Section \ref{example} represents a comparison for the level of difficulty involved in computing even a rank-two perturbation. Recent developments in finite rank perturbations can be found, for example, in \cite{KapPol, LiawTreil0}.

Section \ref{example} contains a simple motivating example, which also . An overview of the constructive process is given in Section \ref{ss-diagram}. And in Section \ref{s-firstpert} we compute location and mass of the eigenvalue generated by the perturbation.

Section \ref{s-startiterate} explains the techniques involved in the iterative construction. Specifically, we fix the first perturbation parameter $\al_1$ and choose the second perturbation vector $\widetilde{\f}_2$ so we can pass via unitary equivalence from the often byzantine a.c.~spectral measure on $[-1,1]$ to an auxiliary measure with constant weight on $[-1,1]$, again. We are mainly concerned with the total mass (or total variation) of the a.c.~part of this auxiliary measure. This auxiliary measure is comparable to the starting measure and is unchanged through the spectral theorem and the unitary operator. This unitary operator and choice of the vector $\widetilde{\f}_2$ return us to the situation at the beginning of Section \ref{s-firstpert}, with a constant measure on $[-1,1]$. 

Section \ref{s-infinite} iterates this utilization of vector choices and unitary operators along with the perturbations. New perturbation directions are orthogonal to the point masses created from previous ones and therefore remain unchanged; this allows us to focus on the absolutely continuous spectrum. Results similar to those from Section \ref{s-startiterate} are produced and the process can effectively be iterated. The final formulas obtained from iteration are found in Subsections \ref{ss-fk} and \ref{ss-tau}. Subsection \ref{ss-rademacher} contains an application of the analysis to the constructed operator with Rademacher potential, where the $\al$'s are chosen to be the endpoints of the given interval, each occurring with equal probability. This operator is found to have spectral localization. These formulas are quite simple and shed further light onto the recursive nature of the process.

Section \ref{s-northo} attempts to escape the requirement of the previous construction of orthogonal perturbation vectors, and obtains results for how much absolutely continuous spectrum can be destroyed via a single rank-one perturbation. No restriction on the direction of the perturbation is made whatsoever. These obtained estimates are upper-bounds and require knowledge of how the perturbation vector interacts with the spectral measure. The goal is to bring the constructed operator closer to being an Anderson-type Hamiltonian by allowing more freedom for the choices of the vectors. Unfortunately, the estimates obtained are not sharp enough to iterate using the devised methods and further refinement is still required. However, the estimates are the first known of their kind for general perturbation theory and are of interest for those purposes as well. The methods used rely on an intimate knowledge of Aronszajn--Donoghue theory and the integral transforms involved within.

%%%%%%%%%%%%%%%%%%%%%%%%%%%%%%%
%%%%%%%%%%%%%%%%%%%%%%%%%%%%%%%
\section{Fundamentals of Perturbation Theory}\label{s-B}
%%%%%%%%%%%%%%%%%%%%%%%%%%%%%%%
%%%%%%%%%%%%%%%%%%%%%%%%%%%%%%%

%%%%%%%%%%%%%%%%%%%%%%%%%%%%%%%
%%%%%%%%%%%%%%%%%%%%%%%%%%%%%%%
\subsection{Classical Perturbation Theory}\label{ss-Pert}
%%%%%%%%%%%%%%%%%%%%%%%%%%%%%%%
%%%%%%%%%%%%%%%%%%%%%%%%%%%%%%%

In perturbation theory one seeks to answer the question:
Given some information about the spectrum of an operator $A$, what can be said about the spectrum of the operator $A+B$ when $B$ is in some operator class?
Often, the attention is restricted to which properties of the spectrum are preserved. The answer, of course, varies wildly depending on the class of operators the perturbation $B$ is taken from. The answer may also be influenced by the choice of unperturbed operator $A$. Here, we focus on self-adjoint operators $A$ and $B$.

To formulate some partial answers, we use the notation 
$$A\sim B (\Mod \text{\em Class }X)$$ 
if there exists a unitary operator $U$ such that $UAU^{-1}-B$ is an element of $\text{\em Class }X$. The $\text{\em Class }X$ can be any class of operators, e.g.~compact, trace class, or finite rank operators. 

\begin{theo}[Weyl--von Neumann, see e.g.~\cite{t-KR}]\label{t-weylvn}
The essential spectra of two self-adjoint operators $A$ and $B$ satisfy 
\begin{align*}
\sigma\ti{ess}(A)=\sigma\ti{ess}(B) \text{ if and only if } A\sim B~(\Mod \text{compact operators}).
\end{align*}
\end{theo}

\begin{theo}[Kato--Rosenblum, see e.g.~\cite{t-KR}]
\label{t-KR}
If two self-adjoint operators satisfy $A\sim B (\Mod \text{trace class})$ then their absolutely continuous parts are unitarily equivalent: $A\ti{ac}\sim B\ti{ac}$.
\end{theo}

\begin{rem*}
For self-adjoint $A$ and $B$, Carey and Pincus \cite{CP} found a complete characterization of when $A\sim B~(\Mod \text{\em trace class})$
in terms of the operators' spectrum. Of course, they must have unitarily equivalent absolutely continuous parts. Outside the continuous spectrum, they are only allowed discrete parts. And the discrete eigenvalues of $A$ and $B$ (counting multiplicity) must fall into three categories: (i) those eigenvalues of $A$ with distances from the joint continuous spectrum having finite sum (i.e.~are trace class), (ii) those eigenvalues of $B$ with distances from the joint continuous spectrum having finite sum, and (iii) eigenvalues of $A$ and $B$ that can be matched up so that their differences have finite sum.
\end{rem*}

%%%%%%%%%%%%%%%%%%%%%%%%%%%%%%%
%%%%%%%%%%%%%%%%%%%%%%%%%%%%%%%
\subsection{Introducing Rank-One Perturbations and the Spectral Theorem}\label{ss-ROST}
%%%%%%%%%%%%%%%%%%%%%%%%%%%%%%%
%%%%%%%%%%%%%%%%%%%%%%%%%%%%%%%

We focus our attention on when the perturbation class $\text{\em Class }X$ consists of self-adjoint operators with one-dimensional range (rank-one). Let $T$ be a self-adjoint operator (bounded or unbounded) on a separable Hilbert space $\cH$. The operator $T$ will be called cyclic when it possesses a vector $\f$ such that
\begin{align}\label{e-cyclicity}
\cH = \overline{\spa\{(T-\lambda{\bf I})^{-1}\f: \lambda\in\CC\backslash\RR\}},
\end{align}
where the closure is taken with respect to the Hilbert space norm. In this case, the vector $\f$ is also called \emph{cyclic}. Here we take $\f\in \cH$. All {\em rank-one perturbations of a self-adjoint operator $T$ by the cyclic vector $\f$} are given by
\begin{align}\label{e-rk1}
T_{\alpha} = T + \alpha \langle \fdot, \f \rangle\ci{\cH}\f
\qquad \text{for}\qquad
\alpha \in \RR.
\end{align}

The supposition that $T$ is cyclic is not a restriction, as otherwise we simply decompose $\cH = \cH_1\oplus \cH_2$ such that $\f$ is cyclic for $T$ on $\cH_1$ and $T$ is left unchanged by the perturbation when restricted to $\cH_2$.

It is worth emphasizing that Theorems \ref{t-KR} and \ref{t-weylvn} can be applied to rank-one perturbations, as such perturbations are both trace class and compact. 

As simple consequence of the resolvent formula one can see $\f$ is also a cyclic vector of the operator $T_{\alpha}$ for all $\alpha\in \R$, see \cite{AbukamovPoltoLiaw, LiawTreil0} for more about cyclicity.
The spectral measure of $T_{\alpha}$ with respect to the cyclic vector $\f$ will be denoted by $\mu_{\alpha}$. Explicitly, the spectral theorem defines $\mu_\alpha$ via 
$$\langle(T_{\alpha}-z\OID)^{-1}\f,\f\rangle\ci{\cH}=\int_\R \frac{d\mu_\alpha(t)}{t-z}\qquad\text{for all }z\in \C\backslash\R.$$
In other words, $T_\alpha$ is unitarily equivalent to multiplication by the independent variable on an $L^2(\mu_{\al})$ space with non-negative Radon measure $\mu_{\al}$, the \emph{spectral measure}, supported on $\R$. The spectral measure of the unperturbed operator $T$, $\mu_0$, is often used as a comparison to the spectral measure $\mu_{\alpha}$. Therefore, we use the convention that $\mu_0=\mu$ for simplicity. This means that $T$ can be written as $M_t$, multiplication by the independent variable on $L^2(\mu)$. The vector $\f$ is then represented by the function that is identically equal to the constant function one on $L^2(\mu)$. 

The spectral theorem now translates the rank-one perturbation problem to
\begin{equation}
\label{d-rk1}
\widetilde{T}_{\alpha}=M_t+\al\langle\fdot, {\bf 1}\rangle\ci{L^2({\mu})}{\bf 1}.
\end{equation} 
Therefore, we identify $\cH=L^2(\mu)$ and use $\widetilde{T}_{\al}=T_{\al}$ for brevity of notation. The presence of a different unitary intertwining operator relating the operators $T_{\al}$ and $M_s$, on their respective spaces $\cH=L^2(\mu)$ and $L^2(\mu_{\al})$, begs the question whether we can capture this unitary intertwining operator. This question was answered in a paper of Liaw and Treil \cite[Theorem 2.1]{LiawTreil1}. The theorem extends to all of $L^2(\mu)$, see  \cite[Theorem 3.2]{LiawTreil1}, but a simpler version is presented here.
\begin{theo}[Representation Theorem]\label{t-repthm}
The spectral representation $V_{\al}:L^2(\mu)\to L^2(\mu_{\al})$ of $T_{\al}$ acts by 
$
V_{\al}f(s)=f(s)-\al\int\ci\R\dfrac{f(s)-f(t)}{s-t}d\mu(t)
$ for compactly supported $C^1$ functions $f$.
\end{theo}

%%%%%%%%%%%%%%%%%%%%%%%%%%%%%%
%%%%%%%%%%%%%%%%%%%%%%%%%%%%%%
\subsection{The Borel Transform and Rank-One Perturbation Theory}\label{ss-ROT}
%%%%%%%%%%%%%%%%%%%%%%%%%%%%%%
%%%%%%%%%%%%%%%%%%%%%%%%%%%%%%

A review of rank-one perturbation theory requires a subtle description of spectral measures and their various decompositions. A study of the integral operators involved will be central to the analysis of Aronszajn--Donoghue theory in Section \ref{s-northo}. Let $\mu$ be a positive measure on $[a,\infty)$ for some $a>-\infty$ with
\begin{align}\label{d-preborel}
\int\dfrac{d\mu(\lambda)}{|\lambda|+1}<\infty.
\end{align}
This assumption is somewhat restrictive, but is necessary for the study of Borel transforms. The condition that the support of $\mu$ is bounded below can be relaxed somewhat, but does hold in the current applications, and simplifies further details slightly.

Adherence to \eqref{d-preborel} allows us to define the \emph{Borel transform} of $\mu$ as
\begin{align*}
F(z):=\int_{\RR}\dfrac{d{\mu}(\la)}{\la -z}\qquad (z\in \C\backslash (\supp\mu)).
\end{align*} 
Indeed, boundary values of $F(z)$, as $z=x+i\epsilon$ approaches points $x$ in the support of $\mu$, are the primary instrument to discern spectral properties of $\mu$. See \cite{CimaMathesonRoss, LiawTreil1, poltsara2006, Simon} for a more detailed discussion (the Cauchy transform, a close relative, is often used).

The auxiliary transform
\begin{align*}
G(x):=\int_{\RR}\dfrac{d\mu(y)}{(y-x)^2}\qquad (x\in \R\backslash (\supp\mu)),
\end{align*}
captures some properties of the derivative (with respect to $z$) of the Borel transform as $z$ approaches the real axis, whereby it also plays a central role.

Let $w\in L^1\ti{loc}(\RR)$ denote the Radon--Nikodym derivative of $\mu$ with respect to Lebesgue measure. With this, the Lebesgue/Radon--Nikodym decomposition is given by $d\mu = w(x) dx + d\mu\ti{s}$. The unitary equivalence between $T$ on $\cH$ and $M_t$ on $L^2(\mu)$ involves a unitary intertwining operator, which gives rise to the corresponding orthogonal components of the operator $T = T\ti{ac}\oplus T\ti{s}$.
The singular part can be further decomposed into singular continuous $\mu\ti{sc}$ and pure point $\mu\ti{pp}$ parts. Here, $\mu\ti{pp}$ consists of point masses at the eigenvalues of $T$ and $\mu\ti{sc}=\mu\ti{s}-\mu\ti{pp}$. The spectrum is denoted by $\sigma(T)$ and is the (closed) $\supp(\mu)$. The set of all real numbers $x$ that are isolated eigenvalues of finite multiplicity for $T$ is defined to be the discrete spectrum, denoted $\sigma\ti{d}(T)$. The essential spectrum of $T$ is the complement of the discrete spectrum, denoted $\sigma\ti{ess}(T)=\sigma(T)\backslash\sigma\ti{d}(T)$.

Historically, the following theorem emerged from the question of changing boundary conditions in a Sturm--Liouville operator \cite{Aron, Dono}. From a theoretical perspective, the theorem characterizes the perturbed operator's pure point and absolutely continuous spectra. The result will be heavily used in later sections.

\begin{theo}[Aronszajn--Donoghue, see e.g.~\cite{Simon}]
\label{t-AD}
For $\al\neq 0$,
define
$$S_\al=\left\{x\in\RR ~|~ F(x+i0)=-\al^{-1}; G(x)=\infty\right\},$$
$$P_\al=\left\{x\in\RR ~|~ F(x+i0)=-\al^{-1}; G(x)<\infty\right\},$$
$$L=\left\{x\in\RR ~|~ \IM~F(x+i0)\neq 0\right\}.$$
Then we have
\begin{enumerate}
\item $\left\{S_{\al}\right\}_{\al\neq 0}$, $\left\{P_{\al}\right\}_{\al\neq 0}$ and $L$ are mutually disjoint.
\item $P_{\al}$ is the set of eigenvalues of $A_{\al}$. In fact,
$$\left(d\mu_{\al}\right)\ti{pp}(x)=\sum_{x_n\in P_{\al}}\dfrac{1}{\al^2 G(x_n)}\delta(x-x_n).$$
\item $\left(d\mu_{\al}\right)\ti{ac}$ is supported on $L$, $\left(d\mu_{\al}\right)\ti{sc}$ is supported on $S_{\al}$.
\item For $\al\neq\beta$, $\left(d\mu_{\al}\right)\ti{s}$ and $\left(d\mu_{\beta}\right)\ti{s}$ are mutually singular.
\end{enumerate}
\end{theo}

The case $\al=\infty$ is known as infinite coupling, and was treated by Gesztesy and Simon, see e.g.~\cite{GS, Simon}.
The last part of the result says that the singular part of rank-one perturbations must move when the perturbation parameter $\alpha$ is changed. A description of the singular continuous spectrum is still outstanding. In fact, the `minimal' support of $\left(\mu_{\al}\right)\ti{sc}$ is not known, see e.g.~\cite{CimaMathesonRoss, LiawTreil1, Simon}, let alone a characterization of $\left(\mu_{\al}\right)\ti{sc}$.

The absolutely continuous part of the perturbed operator, $(\mu_{\al})\ti{ac}$, can be explicitly computed using the following Lemma. 
\begin{lem}[see e.g.~\cite{Simon}]
\label{l-SIM}
Let $F(z)$ be the Borel transform of a measure $\mu$ obeying \eqref{d-preborel}. Let $x\in\RR$ and $\displaystyle x+i0=\lim_{\beta\downarrow 0}(x+i\beta)$. Then we have
\[\IM~F\ci{\alpha}(z)=\dfrac{\IM~F(z)}{|1+\al F(z)|^2} \qquad\text{and} \qquad (d\mu_{\al})\ti{ac}(x)=\pi^{-1}\IM~F_{\al}(x+i0)dx.
\]
\end{lem}

In the case of purely singular measures the following theorem resembles a characterization for $A\sim B (\Mod \text{\em rank-one})$.
\begin{theo}[Poltoratski \cite{Polto2}, also see 
\cite{NPPOL}]\label{t-Polt}
Let $X\subset\R$ be closed. By $I_1=(x_1;y_1), I_2=(x_2;y_2), \hdots$ denote disjoint open intervals such that $X=\R\backslash \bigcup I_n$. Let $A$ and $B$ be two cyclic self-adjoint completely non-equivalent operators with purely singular spectrum. Suppose $\sigma(A)=\sigma(B)=X$ and assume  $\sigma\ti{pp}(A)\cap\{x_1,y_1, x_2, y_2, \hdots\}=\sigma\ti{pp}(B)\cap\{x_1,y_1, x_2, y_2, \hdots\}=\varnothing.$ Then we have $$A\sim B (\Mod \text{\em rank-one}).$$
\end{theo}

%%%%%%%%%%%%%%%%%%%%%%%%%%%%%%%
%%%%%%%%%%%%%%%%%%%%%%%%%%%%%%%
\subsection{Anderson-type Hamiltonians}\label{ss-01}
%%%%%%%%%%%%%%%%%%%%%%%%%%%%%%%
%%%%%%%%%%%%%%%%%%%%%%%%%%%%%%%

Let $(\Omega, \mathcal{A}, \p)$ be a probability space, and consider the sequence of independent random complex variables $X_n(w)$, $w\in \Omega$. We assume that $\Omega=\prod_{n=0}^\infty \Omega_n$, where $\Omega_n$ are different
probability spaces, $w=(w_1,w_2,...)$, $w_n\in\Omega_n$ and the probability measure on
$\Omega$ is introduced as the product measure of the corresponding measures on $\Omega_n$.
Each of the independent random variables $X_n(w)$ depends only on the $n$-th coordinate, $w_n$.

%It is a standard observation that any sequence of independent random variables on
%an abstract probability space is similar to such a sequence $X_n$ defined, for instance, on
%an infinite dimensional torus: So without loss of generality we have $\Omega=\prod_{n=0}^\infty \Omega_n$, where each $\Omega_n$
%is a copy of the unit circle with normalized Lebesgue measure, see e.g.~\cite{K} (where
%the unit interval was used instead of $\T$).

%It is well-known that the properties we are interested in (cyclicity, etc.)~are in fact an event, i.e.~that the set $A$ of $w$, such that the function
%corresponding to the sequence $\{X_n(w)\}$ satisfies the desired
%property, is measurable: $A\in\mathcal{A}$. We will be mostly interested in
%the events $A$ that do not depend on the values of any finite number
%of variables $X_n$, i.e.~the sets $A\in\mathcal{A}$ with the property that
%if $w\in A$ and $X_n(w)=X_n(w')$ for all but finitely many $n$ then
%$w'\in A$. By the zero-one law the probability of any such event is
%0 or 1.

Now, consider a self-adjoint operator $H$ on a separable Hilbert space $\cH$ and a sequence $\{\f_n\}\subset\cH$ of linearly independent unit vectors. Let $\omega=(\omega_1, \omega_2, \hdots)$ be a random variable corresponding to a probability measure $\p$ on $\R^\infty$. In particular, let the parameters $\omega_n$ be chosen i.i.d.~(independent, identically distributed) with respect to $\p$.
An \emph{Anderson-type Hamiltonian} \cite{JakLast2000} is an almost surely self-adjoint operator associated with the formal expression
\begin{equation}\label{Model}
H_\omega = H + V_\omega \qquad\text{on }\cH, \qquad V_\omega = \sum\limits_n \omega_n \langle\fdot, \f_n\rangle\f_n.
\end{equation}
As is customary, assume that the vectors $\f_n$ are orthogonal. However, many properties readily extend to the case of non-orthogonal $\f_n$ so long as \eqref{Model} almost surely defines a self-adjoint operator.

The archetype Anderson-type Hamiltonian is the discrete Schr\"odinger operator with random potential on $l^2(\Z^d)$, given by
\begin{equation}\label{e-dso}
Hf(x)=-\bigtriangleup f (x) = - \sum\limits_{|n|=1} (f(x+n)-f(x)), \quad \f_n(x)=\delta_n(x)=
\left\{\begin{array}{ll}1&x=n,\\ 0&\text{else.}\end{array}\right.
\end{equation}
The corresponding operator $H_\omega$ models quantum mechanical phenomena in a crystalline structure with random on-site potentials, and appears in many fields of mathematics. We note that Kolmogorov's 0-1 Law can be applied to Anderson-type Hamiltonians using the standard probabilistic set up described above, see \cite{AbaPolt, K} where the reader can find all the necessary definitions and basic properties. Many interesting properties that are studied (i.e.~cyclicity) are events, and the 0-1 Law hence states that the probability of such events are either 0 or 1.

%\begin{obs}[Kolmogorov's 0-1 law applied to Anderson-type Hamiltonians]\label{obs}
%Consider the Anderson-type Hamiltonian $H_\omega$ given by \eqref{Model}. Assume that the probability distribution $\p$ satisfies the 0-1 law. Then those spectral properties that are invariant under finite rank perturbations are enjoyed by $H_\omega$ almost surely or almost never.
%\end{obs}

%%%%%%%%%%%%%%%%%%%%%%%%%%%%%%%
%%%%%%%%%%%%%%%%%%%%%%%%%%%%%%%
\section{Motivating Example}\label{example}
%%%%%%%%%%%%%%%%%%%%%%%%%%%%%%%
%%%%%%%%%%%%%%%%%%%%%%%%%%%%%%%

We begin our endeavors by offering a simple example, which serves two purposes. First, it helps demonstrate the difficulties of computing the absolutely continuous part using a scalar spectral measure for even a simple rank--two perturbation, without using the construction introduced in Sections \ref{s-startiterate} and \ref{s-infinite}. Second, we will hook back into this example when we explain the choice of the direction vectors (later called $\f_k$) in the iteration process in Subsection \ref{ss-fk}.

We consider the following rank-two problem in the spectral representation with respect to one of the vectors and assume that the spectral measure of the unperturbed operator is
$d\mu(x):=\frac{1}{2}\chi\ci{[-1,1]}(x)dx$. Consider the normalized vectors 
$\f_1,\f_2\in L^2(\mu)$, where $\f_1$ is the constant function that is identically equal to $\ID$ with $\mu$ respect to almost everywhere.

We introduce the rank-two perturbation
$$H_{\alpha,\beta}=M_t+\alpha\langle\fdot ,\f_1\rangle_{L^2(\mu)}\f_1+\beta\langle\fdot ,\f_2\rangle_{L^2(\mu)}\f_2
\qquad
\text{on}
\qquad
L^2(\mu).$$
As an application of Aronszajn--Donoghue theory we can compute the absolutely continuous part of the rank-one perturbation $H_{\alpha,0}$ for $\beta = 0 $:
\begin{align}
\label{e-r1}
d[(\mu_{\alpha,0})_{ac}](x)= \frac{1}{2}\left\{1+\alpha^2+\alpha\ln\Big(\dfrac{x+1}{-x+1}\Big)+\left(\dfrac{\alpha}{4}\right)^2\Big[\ln\Big(\dfrac{x+1}{-x+1}\Big)\Big]^2\right\}^{-1}dx
\end{align}
for $x\in [-1,1]$, and $(\mu_{\alpha,0})_{ac}\equiv 0$ outside $[-1,1]$.
In general, the introduction of a second rank-one perturbation causes the problem to be too expensive to merit computation. Indeed, Aronszajn--Donoghue theory will require integration with respect to $\mu_{\alpha, 0}$. As a consequence, computing the spectral measure under such an iterative rank-two perturbation, or even just its eigenvalues, seems practically unfeasible.

%%%%%%%%%%%%%%%%%%%%%%%%%%%%%%%
%%%%%%%%%%%%%%%%%%%%%%%%%%%%%%%
\section{Overview of Iterative Process}\label{ss-diagram}
%%%%%%%%%%%%%%%%%%%%%%%%%%%%%%%
%%%%%%%%%%%%%%%%%%%%%%%%%%%%%%%

The general construction described in Sections \ref{s-firstpert} through \ref{s-infinite} is somewhat complicated, so it may be beneficial to the reader to have an overview of the string of processes before diving in. 

We focus on the three main operations and specifically how they change the space we are analyzing. The process begins with a perturbed operator that acts on a space $L^2(\widetilde{\mu}_{\al_{k-1}})$, where $d\widetilde{\mu}_{\al_{k-1}}=\tau_{k-1}\chi\ci{[-1,1]}(x)dx$ for some constant $0<\tau_{k-1}\leq 1$. These $L^2$ spaces, given by a measure with a tilde, are starting points because the total \emph{strength} of the absolutely continuous spectrum (within $[-1,1]$) is easily calculated. They are referred to as auxiliary spaces throughout the manuscript. A broad description of each operation follows, see the referenced sections for more on each step.

\begin{enumerate}
    \item The spectral theorem is used on the operator that acts on $L^2(\widetilde{\mu}_{\al_{k-1}})$ to yield the operator that is multiplication by the independent variable $M_t$ on a new space denoted $L^2(\mu_{\al_{k}})$. The explicit transform is given by the spectral representation $V_{\al_k}$ (the unitary operator realizing the spectral theorem) and described in Subsection \ref{ss-5.1}.
    \item The new operator is perturbed by a parameter $\al_{k+1}$ and specific vector $\f_{k+1}$ given in Corollary \ref{c-morevectors}, that depends on the previous perturbation. Vector $\f_{k+1}$ is orthogonal to previous perturbation vectors, allowing us to focus on just the absolutely continuous part of the measure $L^2[(\mu_{\al_k})\ti{ac}]$. For convenience, we denote this operation as ``$+\f_{k+1}$'' over a squiggly arrow in the diagram below.
    \item Finally, a unitary transform $U_k$ is applied to the operator which translates it to the space $L^2(\widetilde{\mu}_{\al_{k}})$, see Corollary \ref{c-morevectors}. This space is comparable to $L^2(\widetilde{\mu}_{\al_{k-1}})$, and our operations may be repeated again.
\end{enumerate}

The following schematic details how the $k+1$ and $k+2$ perturbations are performed by identifying the spaces of interest.

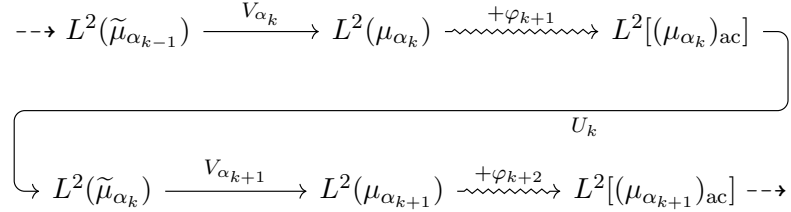
\begin{center}
\begin{tikzcd}[row sep=huge, column sep=large]
 \dashrightarrow L^2(\widetilde{\mu}_{\al_{k-1}}) \arrow[r,"V_{\al_{k}}"]
    & L^2(\mu_{\al_{k}}) \arrow[r, rightsquigarrow,"+\f_{k+1}"]
        \arrow[d, phantom, ""{coordinate, name=Z}]
      & L^2[(\mu_{\al_k})\ti{ac}] \arrow[dll,
                 "U_k",
rounded corners,
to path={ -- ([xshift=2ex]\tikztostart.east)
|- (Z) [near end]\tikztonodes
-| ([xshift=-2ex]\tikztotarget.west) -- (\tikztotarget)}] \\
         L^2(\widetilde{\mu}_{\al_k}) \arrow[r,"V_{\al_{k+1}}"]
    & L^2(\mu_{\al_{k+1}}) \arrow[r, rightsquigarrow,"+\f_{k+2}"]
& L^2[(\mu_{\al_{k+1}})\ti{ac}] \dashrightarrow \end{tikzcd}
\end{center}

Please see the beginning of Section \ref{s-infinite} for a more detailed procedure.

%%%%%%%%%%%%%%%%%%%%%%%%%%%%%%%
%%%%%%%%%%%%%%%%%%%%%%%%%%%%%%%
\section{A First Perturbation}\label{s-firstpert}
%%%%%%%%%%%%%%%%%%%%%%%%%%%%%%%
%%%%%%%%%%%%%%%%%%%%%%%%%%%%%%%

We begin with constructing a rank-one perturbation in the spectral representation. Namely, consider the spectral measure 
\begin{align}\label{e-firstmeasure}
d\widetilde{\mu}_0(x):=\dfrac{1}{2}~\chi\ci{[-1,1]}(x)dx
\end{align}
and let vector $\widetilde\f_1$ be the constant function of $L^2(\widetilde\mu_0)$ that is identically equal to 1. Note that $\widetilde\f_1$ has unit norm.
Now, consider the family of (bounded) rank-one perturbations:
\begin{align}\label{e-firstpert}
\widetilde{H}\ci{\alpha_1}=M_t+\alpha_1\langle\fdot ,\widetilde{\f}_1\rangle_{L^2(\widetilde{\mu}_0)}\widetilde{\f}_1
\qquad
\text{on }
L^2(\widetilde{\mu}_0)\text{, where }
\alpha_1\in\R.
\end{align}

The total mass of a measure $\eta$,
\begin{align*}
\|{\eta}\|:=\int_{\RR}d{\eta}(t),
\end{align*}
will play a key role in comparing the remaining mass of the absolutely continuous parts of the spectral measures we produce within the iteration process.

By construction we have $\|\widetilde{\mu}_0\|=1$. The properties of the perturbed operator are captured by Aronszajn--Donoghue theory (Theorem \ref{t-AD}). Applied to the current situation, this result yields the following observation, which we will utilize at each step of our construction.

\begin{lem}\label{l-compute} 
Let $\widetilde{\mu}_0$ and $\widetilde{H}_{\al_1}$ respectively be given by \eqref{e-firstmeasure} and \eqref{e-firstpert}. Let $\mu\ci{\alpha_1}$ be the spectral measure of $\widetilde{H}\ci{\alpha_1}$. If $\alpha_1 \neq 0$, then

\begin{enumerate}
\item the perturbed operator $\widetilde{H}\ci{\alpha_1}$ has exactly one eigenvalue, $x\ci{\alpha_1}$, and
\begin{align*}
x\ci{\alpha_1}:=\dfrac{-1-e^{2/\al_1}}{1-e^{2/\al_1}}\,\in \R\backslash [-1,1].
\end{align*}
\item The created eigenvalue has weight/mass:
\begin{align*}
\mu\ci{\alpha_1}\{x\ci{\alpha_1}\}=\dfrac{4 e^{2/\al_1}}{\al_1^2(e^{2/\al_1}-1)^2}\,.
\end{align*}

\end{enumerate}
\end{lem}

\begin{proof}
\indent $(1)$ On $[-1,1]$, the imaginary part of $F(z)=\int\ci{\mathbb{R}}\frac{d\widetilde{\mu}_0(\lambda)}{\lambda -z}$ is strictly positive. So, $
\widetilde{H}\ci{\alpha_1}$ does not have any eigenvalues on $[-1,1]$ for all $\alpha_1$.

Further, the assumptions of the lemma imply
\begin{align*}
G(x)=\int_{-1}^{1}\dfrac{d\widetilde{\mu}_0(\lambda)}{(\lambda -x)^2}<\infty
\qquad
\text{for all }x\in\R\backslash [-1,1].
\end{align*}
By Theorem \ref{t-AD},  
$\widetilde{H}\ci{\alpha_1}$ has an eigenvalue at such $x$ if and only if
$-\dfrac{1}{\alpha_1}=\text{F}(x+i0)$.
Hence, eigenvalues occur at $x\in\R\backslash [-1,1]$ that satisfy
\begin{align*}
-\dfrac{1}{\alpha_1}=\int\ci{\mathbb{R}}\dfrac{d\widetilde{\mu}_0(\lambda)}{\lambda - x}=\dfrac{1}{2}\int_{-1}^{1}\dfrac{d\lambda}{\lambda - x}=\dfrac{1}{2}\text{ln}\left(\dfrac{1-x}{-1-x}\right).
\end{align*}
The solution to the previous equation for $x$ depends on $\al_1$ and will be denoted as
\begin{align*}
x_{\alpha_1}:=\dfrac{-1-e^{2/\al_1}}{1-e^{2/\alpha_1}}\,.
\end{align*}
In particular, $x_{\alpha_1}<-1$ for $\alpha_1<0$, while $x_{\alpha_1}>1$ for $\alpha_1>0$.

$(2)$ By Theorem \ref{t-AD}, the mass of the created eigenvalue is
$
\mu\ci{\alpha_1}\{x\ci{\alpha_1}\}=\dfrac{1}{\al_1^2 G(x\ci{\alpha_1})},
$
where
\begin{align*}
G(x\ci{\alpha_1})=\dfrac{1}{2}\int_{-1}^{1}\dfrac{d\la}{(\la -x\ci{\alpha_1})^2}=-\dfrac{1}{2}\left[\dfrac{1}{1-x\ci{\alpha_1}}+\dfrac{1}{1+x\ci{\alpha_1}}\right].
\end{align*}
Inserting the value of $x\ci{\alpha_1}$ calculated in part $(1)$ yields $1/G(x\ci{\alpha_1})=\dfrac{4 e^{2/\al_1}}{(e^{2/\al_1}-1)^2}$. The second statement of the lemma thus follows.
\end{proof}

%%%%%%%%%%%%%%%%%%%%%%%%%%%%%%%
%%%%%%%%%%%%%%%%%%%%%%%%%%%%%%%
\section{Iterated Perturbations}\label{s-startiterate}
%%%%%%%%%%%%%%%%%%%%%%%%%%%%%%%
%%%%%%%%%%%%%%%%%%%%%%%%%%%%%%%

This section explains the heart of our construction. After fixing $\al_1$, we set the stage for the successive perturbations by describing how the operator $H\ci{\alpha_1}$ is perturbed. In particular, the specific choice of the direction of the second perturbation, $\f_2$, will allow us to calculate the total mass of the remaining absolutely continuous part of the spectrum. The difficulties encountered in the example in Section \ref{example} are bypassed by applying a unitary transformation which exploits the choice of $\f_2$. After the transformation, computations from Aronszajn--Donoghue theory again resemble those of Lemma \ref{l-compute}.

\subsection{Unitary equivalence and the remaining absolutely continuous spectrum}\label{ss-5.1} Recall the rank-one perturbation setup discussed in Section \ref{s-firstpert}, namely,
\begin{align*}
\widetilde{H}\ci{\alpha_1}=M_t+\alpha_1\langle\fdot ,\widetilde{\f}_1\rangle_{L^2(\widetilde{\mu}_0)}\widetilde{\f}_1
\,\,\,\text{on}\,\,\,
L^2(\widetilde{\mu}_0)
\,\,\,\text{where}\,\,\,
d\widetilde{\mu}_0(x):=\dfrac{1}{2}~\chi\ci{[-1,1]}(x)dx
\,\,\,\text{and}\,\,\,
\widetilde\f_1\equiv {\bf 1}.
\end{align*}

Fix the realization of $\al_1$ in accordance with the probability measure $\p$. Then, Aronszajn--Donoghue theory (Theorem \ref{t-AD}) provides us with information about the spectral measure, $\mu\ci{\alpha_1}$, of the perturbed operator $\widetilde{H}\ci{\alpha_1}$ from the previous section. Furthermore, we know that the support of the absolutely continuous part of the measure is still equal to $[-1,1]$ due to the Kato--Rosenblum theorem, Theorem \ref{t-KR}. The operator $\widetilde{H}\ci{\alpha_1}$ is represented in the space $L^2(\mu\ci{\alpha_1})$ as multiplication by the independent variable due to the spectral theorem. 

By a slight abuse of notation, for future iterations we will still write $M_t$ for this operator to avoid an infinite sequence of independent variables. In particular, we have the unitary equivalence between operators
\begin{align}\label{e-SpecRep}
\left(\widetilde{H}\ci{\alpha_1}\text{ on }L^2(\widetilde{\mu}_0)\right)
\,\,\,\,\sim\,\,\,\,
\left(M_t\text{ on }L^2({\mu}\ci{\alpha_1})\right).
\end{align}
Let $V\ci{\alpha_1}: L^2(\widetilde{\mu}_0)\to L^2({\mu}\ci{\alpha_1})$ denote the corresponding intertwining unitary operator such that
$$
V\ci{\alpha_1} \widetilde{H}\ci{\alpha_1}= M_t V\ci{\alpha_1}
\qquad\text{and}\qquad
V\ci{\alpha_1}\widetilde\f_1 \equiv \text{constant}.
$$

%Of course, the value of the constant can be determined by normalizing $V\ci{\alpha_1}\widetilde\f_1$. 
In order to find the value of this constant, we notice that the operator $V\ci{\al_1}$ is given explicitly in the Representation Theorem \ref{t-repthm} \cite{LiawTreil1}. Hence, by construction we have that $V\ci{\al_1}\ID=\ID$, where the $\ID$ vectors are understood to be in the appropriate $L^2$ spaces, $L^2(\widetilde{\mu}_0)$ and $L^2(\mu\ci{\al_1})$ respectively.

Together with the unitary property of $V_{\al_1}$ we see that
\begin{align}\label{calc}
\|\widetilde{\mu}_0\|=\int_{\RR}d{\widetilde{\mu}_0}(t)
%=&\int_{\RR}\ID d{\widetilde{\mu}_0}(t)
=\|\ID\|\ci{L^2(\widetilde{\mu}_0)}=\|V\ci{\al_1}\ID\|\ci{L^2(\mu_{\al_1})} %\\
&=\|\ID\|\ci{L^2(\mu_{\al_1})}%=\int_{\RR}\ID d{\mu_{\al_1}}(t)
=\int_{\RR}d{\mu_{\al_1}}(t)=\|\mu_{\al_1}\|.
\end{align}

In particular, we have
\[
\|(\mu_{\al_1})\ti{ac}\|
=
\|\widetilde{\mu}_0\| - \mu\ci{\alpha_1}\{x\ci{\alpha_1}\}
=\|(\widetilde{\mu}_0)\ti{ac}\| - \mu\ci{\alpha_1}\{x\ci{\alpha_1}\}
=1-\mu\ci{\alpha_1}\{x\ci{\alpha_1}\}<1.
\]

\begin{rem*}
Researchers experienced in the field may feel \eqref{calc} to be contradictory to results in Clark theory, the basis of unitary rank-one perturbation theory. However, a central theme discovered while producing these results is that in these aspects, self-adjoint theory and unitary theory are quite different. For instance, attempting to move this calculation into the unitary perturbation case with the Cayley transform involves an adjustment for the perturbation vector which causes some cancellations, see \cite[Lemma 5.1]{LiawTreil2}. Furthermore, the condition in the representation theorem that requires $V\ci{\al_1}\ID=\ID$ is believed to be equivalent to the statement that $\theta(0)=0$, where $\theta$ is the generating characteristic function for the Clark measures. Hence, a contradiction with a result similar to \cite[Prop. 9.1.8]{CimaMathesonRoss} is not created.
\end{rem*}

\subsection{The direction of the second perturbation vector} The second rank-one perturbation is now added in the direction of some particular function $\f_2\in\cH$, to be chosen in accordance with Proposition \ref{p-f2}. The task will be to observe properties of
\begin{equation}\label{e-example}
H\ci{\alpha_2}=M_t+\al_2\langle\fdot,\f_2\rangle\ci{L^2(\mu\ci{\alpha_1})}\f_2\quad\text{on }L^2(\mu\ci{\alpha_1}).
\end{equation}
To see that $H\ci{\alpha_2}$ is a rank-two perturbation of the original operator, recall that by \eqref{e-SpecRep} the operator $M_t$ on $L^2(\mu\ci{\alpha_1})$ is the spectral presentation of the original perturbed operator $H_{\alpha_1}$ on $L^2(\widetilde{\mu}_0)$.

As announced, we study $H\ci{\alpha_2}$ by moving to an auxiliary space. The following proposition encapsulates the main idea of this work: the choice of $\f_2$ and of the unitary operator $U_1$ which passes the spectral calculations from $L^2(\mu_{\al_1})$ to a particular \emph{auxiliary space} denoted by $L^2(\widetilde{\mu}_{\al_1})$.

\begin{prop}\label{p-f2}
By a choice of a unitary multiplication operator $U_1$ and a unit vector $\f_2\in L^2({\mu}\ci{\alpha_1})$, we can arrange for
\begin{align}\label{e-perp}
\f_2\perp L^2[(\mu\ci{\alpha_1})\ti{pp}],\qquad \f_2\perp \ID
\end{align}
(recall that $V\ci{\alpha_1}\widetilde\f_1\equiv\,$constant, so that $\f_2\perp V\ci{\alpha_1}\widetilde\f_1$),
and for the following unitary mapping of spaces
\begin{align}\label{e-U1}
U_1: L^2[(\mu\ci{\alpha_1})\ti{ac}]\to L^2(\widetilde{\mu}\ci{\alpha_1}),
%\qquad
%U_1 H_{\alpha_1} = M_t U_1,
\end{align}
as well as for
\begin{align}\label{e-tau1}
d(\widetilde{\mu}_{\al_1})\ci{ac} (x)= \tau_1\chi\ci{[-1,1]}(x)dx
\qquad \text{for some constant }\tau_1.
\end{align}
%\begin{enumerate}
%\item $\f_2\perp L^2[(\mu\ci{\alpha_1})\ti{pp}],$ as well as $\f_2\perp \ID$ (recall that $V\ci{\alpha_1}\widetilde\f_1\equiv\,$constant, so that $\f_2\perp V\ci{\alpha_1}\widetilde\f_1$),
%\item the following mapping of spaces and intertwining of operators $U_1: L^2[(\mu\ci{\alpha_1})\ti{ac}]\to L^2(\widetilde{\mu}\ci{\alpha_1})$ and $U_1 H_{\alpha_1} = M_t U_1$, and 
%\item $d(\widetilde{\mu}_{\al_1})\ci{ac} = \tau_1\chi\ci{[-1,1]}dx$ for some constant $\tau_1$.
%\end{enumerate}
\end{prop}

%\begin{prop}\label{p-f2}
%The rank-two perturbation $\widetilde{H}\ci{\alpha_2}$ can be recast as
%\begin{equation}\label{e-ranktwo}
%\widetilde{H}\ci{\alpha_2}=M_t+\al_2\langle\fdot,\widetilde{\f}_2\rangle\ci{L^2(\widetilde{\mu}_{\al_1})}\widetilde{\f}_2\quad\text{on }L^2(\widetilde{\mu}_{\al_1}).
%\end{equation}
%By a choice of a unitary multiplication operator $U_1$ and a unit vector $\f_2\in L^2({\mu}\ci{\alpha_1})$, we can arrange for
%\begin{enumerate}
%\item $\f_2\perp L^2[(\mu\ci{\alpha_1})\ti{pp}],$ and
%\item $\f_2\perp \ID$ (recall that $V\ci{\alpha_1}\widetilde\f_1\equiv\,$constant, so that $\f_2\perp V\ci{\alpha_1}\widetilde\f_1$),
%\item $d(\widetilde{\mu}_{\al_1})\ci{ac} = \tau_1\chi\ci{[-1,1]}dx$ for some constant $\tau_1$.
%\end{enumerate}
%\end{prop}

\begin{rem*} Before we prove this proposition, we explain what the results mean for the successive  construction.
\begin{itemize}
\item[(a)] Equation \eqref{e-perp} ensures that the previously introduced eigenvalues remain unchanged and that the sequence of direction vectors will be orthogonal.
\item[(b)] Equations \eqref{e-U1} and \eqref{e-tau1} mean that the problem is once again simplified to one that resembles the setup in Section \ref{s-firstpert}. So, by the spectral theorem, the rank-two perturbation $\widetilde{H}\ci{\alpha_2}$ can be recast as
\begin{equation}\label{e-ranktwo}
\widetilde{H}\ci{\alpha_2}=M_t+\al_2\langle\fdot,\widetilde{\f}_2\rangle\ci{L^2(\widetilde{\mu}_{\al_1})}\widetilde{\f}_2\quad\text{on }L^2(\widetilde{\mu}_{\al_1}).
\end{equation}

 Indeed, observe how the mass of the absolutely continuous spectrum decreases. For instance, let $(\mu\ci{\alpha_2})\ci{ac}$ denote the absolutely continuous part of the spectral measure corresponding to the rank-two perturbed operator $\widetilde{H}\ci{\alpha_2}$. In light of the above discussion on the properties of $V_{\al_1}$, the mass of this measure can be calculated as
\begin{align*}
\|(\mu\ci{\alpha_2})\ci{ac}\|=\|(\mu\ci{\alpha_1})\ci{ac}\|-\mu\ci{\alpha_2}\{x\ci{\alpha_2}\} =1 - \mu\ci{\alpha_1}\{x\ci{\alpha_1}\}-\mu\ci{\alpha_2}\{x\ci{\alpha_2}\}.
\end{align*}
This is the main reason why we make the particular choices for the direction vectors.
Thus, the proposition claims that our task becomes much simpler when we pass to an auxiliary measure $\widetilde{\mu}_{\al_1}$. Indeed, the numerical constant $\tau_1$ is related to those in Lemma \ref{l-compute} via $\tau_1  = \|(\widetilde{\mu}_{\al_1})\ci{ac}\|/2$.
\end{itemize}
\end{rem*}

\begin{proof}[Proof of Proposition \ref{p-f2}]
Recall that the first perturbation only had the effect of creating an eigenvalue, so we may assume that
$$
(d\mu_{\al_1})\ti{ac}(t)=w_1(t)dt,
$$
where $w_1(t)$ is some weight function. As was done in the example in Section \ref{example}, this weight function can be exactly determined by using Theorem \ref{t-AD} and Lemma \ref{l-SIM}. We omit this tedious calculation for brevity. Also, $x\ci{\alpha_1}$ was an eigenvalue outside $[-1,1]$ created by $\f_1$, analyzed in Lemma \ref{l-compute}. The corresponding eigenvector is supported at the single point $\left\{x\ci{\alpha_1}\right\}$.

Define the function $h_2\in L^2[(\mu_{\al_1})\ti{ac}]$ in accordance with Lemma \ref{l-APP} such that 
\begin{align*}
    h_2(x_{\al_1})=0, \text{ } |h_2(x)|=1 \text{ for } x\in[-1,1], \text{ and } h_2\perp V_{\al_1}\widetilde{\f}_1.
\end{align*}
Then define the perturbation vector 
\begin{align*}
\f_2(t)=\left\{
\begin{array}{ll}\dfrac{h_2(t)}{\sqrt{2w_1(t)}}\quad& \text{for }t\in (-1,1),\\0&\text{else.}\end{array}\right.
\end{align*}
This definition ensures that the conditions of equation \eqref{e-perp} are satisfied, thanks to the function $h_2$. In particular, we have the orthogonal decomposition
\begin{equation}\label{e-pertbreakdown}
H\ci{\alpha_2}=M_t\oplus[M_t+\al_2\langle\fdot,\f_2\rangle_{L^2(\mu\ci{\alpha_1})}\f_2]\quad\text{on }L^2[(\mu\ci{\alpha_1})\ti{pp}]\oplus L^2[(\mu\ci{\alpha_1})\ti{ac}],
\end{equation}
so that the eigenvalue $x\ci{\alpha_1}$ will remain unchanged by the second perturbation.
Further examinations can thus be reduced to $M_t+\al_2\langle\fdot,\f_2\rangle\f_2$ on $L^2[(\mu\ci{\alpha_1})\ti{ac}]$. 

The choice of $\f_2$ has several favorable consequences. The spectral representation of $H\ci{\alpha_2}$ with respect to the vector $\f_2$ will be used to transform to the appropriate auxiliary space $L^2(\widetilde{\mu}_{\al_1})$ that will guarantee \eqref{e-U1}. Let us carry this out.

The unitary operator that realizes this spectral representation is the multiplication operator $U_1: L^2[(\mu\ci{\alpha_1})\ti{ac}]\to L^2(\widetilde{\mu}\ci{\alpha_1})$ is given by
\begin{align}
\label{d-U}
U_1:=M_{\sqrt{w_1(t)/\tau_1}}.
\end{align}
Operator $U_1$ is unitary because if $f\in L^2(\mu\ci{\alpha_1})\ti{ac}$, then
\begin{align*}
\|f\|_{L^2[(\mu\ci{\alpha_1})\ti{ac}]}^2&=\int_{-1}^{1}|f(t)|^2(d\mu\ci{\alpha_1})\ti{ac}(t)=\int_{-1}^{1}|f(t)|^2w_1(t)dt \\
&=\int_{-1}^{1}\left|f(t)\sqrt{w_1(t)}\right|^2dt=\bigintss_{-1}^{1}\left|f(t)\sqrt{\dfrac{w_1(t)}{\tau_1}}\right|^2\tau_1 dt =\|U_1 f\|_{L^2(\widetilde{\mu}\ci{\alpha_1})}^2.
\end{align*} 

The mass of $x\ci{\alpha_1}$ was explicitly calculated in Section \ref{s-firstpert}. We define
\begin{align*}
\tau_1:=||(\mu_{\al_1})\ti{ac}||/2=\frac{1}{2}\int_{-1}^{1}(d\mu_{\al_1})\ti{ac}(t)=\frac{1}{2}\left(1-\mu_{\al_1}\{x\ci{\alpha_1}\}\right) 
=\frac{1}{2}\left(1-\dfrac{4 e^{2/\al_1}}{\al_1^2(e^{2/\al_1}-1)^2}\right).
\end{align*}
And the explicit verification of property $(3)$ then follows:
\begin{align*}
    ||(\widetilde{\mu}_{\al_1})\ti{ac}||=||1||_{L^2[(\widetilde{\mu}_{\al_1})\ti{ac}]}&=||U_1^{-1}(1)||_{L^2[(\mu_{\al_1})\ti{ac}]}\\
    &=\int_{-1}^1\left|\sqrt{\dfrac{\tau_1}{w_1(t)}}\right|^2w_1(t)dt=\int_{-1}^1\tau_1dt=2\tau_1,
\end{align*}
where we used the fact that $U_1$, and hence $U_1^{-1}$, is unitary for the second equality.
Finally, we note that $\f_2$ is a unit vector in $L^2[(\mu_{\al_1})\ti{ac}]$:
\begin{align*}
    ||\f_2||_{L^2[(\mu_{\al_1})\ti{ac}]}=\bigintsss_{-1}^1\left|\dfrac{h_2(t)}{\sqrt{2w_1(t)}}\right|^2w_1(t)dt=\int_{-1}^1\dfrac{1}{2}dt=1.
\end{align*}
The unitary operator $U_1$ then establishes $\widetilde{\f}_2$ as a unit vector in $L^2(\widetilde{\mu}_{\al_1})$, in accordance with equation \eqref{e-ranktwo}. It can be seen that with these choices of $U_1$ and $\f_2$, the result follows. 
\end{proof}

%%%%%%%%%%%%%%%%%%%%%%%%%%%%%%%
%%%%%%%%%%%%%%%%%%%%%%%%%%%%%%%
\section{Absolutely Continuous Spectrum under Infinite Iterations}\label{s-infinite}
%%%%%%%%%%%%%%%%%%%%%%%%%%%%%%%
%%%%%%%%%%%%%%%%%%%%%%%%%%%%%%%
The iteration strategy is now essentially clear:
\begin{enumerate}

\item Begin with the setting as in Section \ref{s-firstpert}, that is, consider the multiplication by the independent variable on  $L^2(\widetilde\mu_0)$ with $d\widetilde{\mu}_0(x):=\dfrac{1}{2}~\chi\ci{[-1,1]}(x)dx$ and the perturbation direction $\widetilde\f_1\equiv {\bf 1}$ in $L^2(\widetilde\mu_0)$.

\item Take a probability distribution $\Omega$, and use it to fix a realization $(\alpha_1, \alpha_2, \hdots)$ of the random variables.

\item Carry out the first perturbation as in Subection \ref{ss-5.1}. Let $k=2$.

\item Take the unit vector in direction $\f_k\in L^2(\mu_{\al_{k-1}})$ in accordance with Proposition \ref{p-f2}.

\item By Proposition \ref{p-f2} and \eqref{e-ranktwo} the new perturbation problem $\widetilde H\ci{\alpha_k}$ on the auxiliary space $L^2(\widetilde{\mu}_{k-1})$ is as follows: the unperturbed perturbed operator equals multiplication by the independent variable on the $L^2 $ space with measure $d\widetilde{\mu}_{k-1}(x)= \tau_{k-1} \chi\ci{[-1,1]}(x)dx$ and the perturbation direction of the constant unit vector $\widetilde\f_{k}\in L^2(\widetilde{\mu}_{k-1})$. 

\item Apply the spectral theorem to yield the operator $M_t$ on the space $L^2(\mu_{\al_k})$. This drops the ``\,$\widetilde{\,\,\,}\,$" in the notation and replaces $k$ by $k+1$.

\item Repeat steps (4) through (7) with Proposition \ref{e-ranktwo} replaced by Corollary \ref{c-morevectors} below.
\end{enumerate}

Note that $\tau_{k-1}$ equals the remaining total mass after the $k$-th iteration, see Subsection \ref{ss-tau} below. Keeping track of the sequence $\{\tau_k\}$ is the ultimate goal of our endeavors. Also see Section \ref{ss-diagram} for an overview of the iterated process.

%%%%%%%%%%%%%%%%%%%%%%%%%%%%%%%
%%%%%%%%%%%%%%%%%%%%%%%%%%%%%%%
\subsection{The $k$-th Perturbation Vector}\label{ss-fk}
%%%%%%%%%%%%%%%%%%%%%%%%%%%%%%%
%%%%%%%%%%%%%%%%%%%%%%%%%%%%%%%

After step (1) the problem is to consider the $k$-th rank-one perturbation. In analogy to equation \eqref{e-ranktwo} we now consider
\[
H\ci{\alpha_k}=M_t+\al_k\langle\fdot,\f_k\rangle\ci{L^2(\mu\ci{\alpha_{k-1}})}\f_k\quad\text{on }L^2(\mu\ci{\alpha_{k-1}}).
\]

Let $\{f_1,\hdots,f_{k-1}\}$ denote the vectors in $L^2(\mu\ci{\alpha_{k-1}})$ that correspond to the directions of previous perturbations, which were chosen after the previous $k-1$ steps. In other words, we let 
\[
\left(f_n\in L^2(\mu\ci{\alpha_{n-1}})\right)
\quad
\sim\quad
\left(\f_n\in L^2(\mu\ci{\alpha_{n}})\right)
\qquad\text{for}\qquad
n=1,\hdots,k-1,
\]
where $\sim$ refers to the unitary equivalence in accordance with appropriate composition (different for each $n$) of unitary transformations.
Recall that $M_t$ in $L^2(\widetilde{\mu}_{\al_{k-1}})$ corresponds to the previous rank-$(k-1)$ perturbation in its spectral representation. The following corollary to the proof of Proposition \ref{p-f2} shows that the direction of the $k$-th perturbation vector $\f_k$ can be chosen analogously. 

\begin{cor}\label{c-morevectors}
We can choose $\f_k\in L^2({\mu}\ci{\alpha_{k-1}})$ so that
\begin{align}\label{e-perp2}
\f_k\perp L^2[(\mu\ci{\alpha_n})\ti{pp}]
\text{ and }\f_k\perp f_n\text{ for all }n=1, \hdots, k-1,
\end{align}
and we can choose a unitary multiplication operator $U_{k-1}: L^2[(\mu\ci{\alpha_{k-1}})\ti{ac}]\to L^2(\widetilde{\mu}\ci{\alpha_{k-1}})$.
The rank-$k$ perturbation of interest then becomes
\begin{align}\label{e-rankk}
\widetilde{H}\ci{\alpha_k}=M_t+\al_k\langle\fdot,\widetilde{\f}_k\rangle\ci{L^2(\widetilde{\mu}_{\al_{k-1}})}\widetilde{\f}_k\quad&\text{on }L^2(\widetilde{\mu}_{\al_{k-1}}), 
d\widetilde{\mu}_{k-1}(x)\equiv\tau_{k-1}\chi\ci{[-1,1]}(x)dx.
\end{align}
\end{cor} 
As in the remark after Proposition \ref{p-f2}, we note that $\tau_{k-1}  = m_{k-1}/2 = \|(\widetilde{\mu}_{\al_{k-1}})\ci{ac}\|/2.$

\begin{proof}
We mimic the proof of Proposition \ref{p-f2}. Consider
\begin{align*}
\f_k(t):= \begin{cases} 0 & \text{on }\RR \setminus [-1,1], \\ 
\dfrac{h_k(t)}{\sqrt{2w_{k-1}(t)}} \qquad& \text{on }(-1,1), \end{cases}
\end{align*}
where $(d\mu_{k-1})\ti{ac}(t)=w_{k-1}(t)dt$. The function $h_k(t)$ is such that $|h_k(t)|=1$ on $(-1,1)$, and is chosen by Lemma \ref{l-APP} with
$$
d\eta(t) = \dfrac{d{\mu}\ci{\alpha_{k-1}}(t)}{\sqrt{2w_{k-1}(t)}}\,,
\quad
f_n=f_n, \text{ for }n=1,\hdots, k-1,
\quad
\text{and}
\quad h=h_k.
$$

This implies \eqref{e-perp2}. Define the multiplication operator $U_{k-1}: L^2[(\mu\ci{\alpha_{k-1}})\ti{ac}]\to L^2(\widetilde{\mu}_{k-1})$ to be 
\begin{align*}
U_{k-1}:=M_{\sqrt{w_{k-1}(t)}/h_k(t)}.
\end{align*}
If we denote $U_{k-1}\f_k$ by $\widetilde{\f}_k$, we then have that $\|\widetilde{\f}_k\|_{L^2(\widetilde{\mu}_{k-1})}=1$. Property \eqref{e-rankk} is obtained by the definition of the spectral theorem, as in the remark after Proposition \ref{p-f2}.
\end{proof}

%%%%%%%%%%%%%%%%%%%%%%%%%%%%%%%
%%%%%%%%%%%%%%%%%%%%%%%%%%%%%%%
\subsection{Remaining Absolutely Continuous Spectrum after $k$ Iterations}\label{ss-tau}
%%%%%%%%%%%%%%%%%%%%%%%%%%%%%%%
%%%%%%%%%%%%%%%%%%%%%%%%%%%%%%%

The desired byproduct of this construction is now achieved. The specific choice of $\f_k$ at each step in Corollary \ref{c-morevectors} allows the proof of Lemma \ref{l-compute} to be generalized to each iteration because
 \begin{align*}
 G_k(x)=\int_{\RR}\dfrac{d\widetilde{\mu}_k(t)}{(t-x)^2}=\tau_k\int_{-1}^{1}\dfrac{dt}{(t-x)^2}<\infty
 \quad\text{for}\quad
 x\notin [-1,1].
 \end{align*}
This means that Aronszajn--Donoghue theory applies and the essential formulas from Section \ref{s-startiterate} can be generalized.

Recall that $d\widetilde{\mu}_0(x)=\tau_0\chi\ci{[-1,1]}(x)dx$, with $\tau_0=1/2$ as above. By equation \eqref{e-rankk} we  have
\begin{align*}
(d\widetilde{\mu}_{\al_k})\ci{ac}(x)=\tau_k\chi\ci{[-1,1]}(x)dx.
\end{align*}
We determine $\tau_k$ in a similar way as $\tau_1$ in Section \ref{s-startiterate}. Specifically,
\begin{align*}
\tau_k=\dfrac{\|(\mu_{\al_{k-1}})\ci{ac}\|}{2}-\dfrac{\widetilde{\mu}_{\al_k}\{x\ci{\al_k}\}}{2}.
\end{align*}

Again, the eigenvalue $x\ci{\alpha_k}$, created by the perturbation $\al_k$, is unaffected by subsequent perturbations. However, the calculations for $\widetilde{\mu}_{\al_k}\{x\ci{\al_k}\}$ will involve the constant $\tau_{k-1}$ from the previous step. Hence, the formulas in Section \ref{s-startiterate} are recursive and need to be altered slightly:
\begin{align*}%\label{e-iteratedac}
 \|(\widetilde{\mu}_{\al_k})\ci{ac}\|=\|(\widetilde{\mu}_{\al_{k-1}})\ci{ac}\|-\widetilde{\mu}_{\al_k}\{x\ci{\alpha_{k}}\} 
 =1-\sum_{n=1}^{k}\widetilde{\mu}_{\al_n}\{x_{\al_n}\} 
 \end{align*}
and together with item (2) of Lemma \ref{l-compute}  we have shown:
\begin{prop}\label{p-recursiveac}
The remaining absolutely continuous spectrum after $k$ iteration steps is
\[
 \|(\widetilde{\mu}_{\al_k})\ci{ac}\|=1-\sum_{n=1}^{k}\dfrac{e^{1/\al_n\tau_{n-1}}}{\al_n^2\tau_{n-1}(e^{1/\al_n\tau_{n-1}}-1)^2}\,.
 \]
\end{prop}

 The recursive process used to determine the terms in the sum is illustrated by:
 $$
 \tau_{k-1}\quad\to\quad \widetilde{\mu}_{\al_k}\{x\ci{\al_k}\}\quad\to\quad \|(\widetilde{\mu}_{\al_k})\ci{ac}\| = 2 \tau_k.
 $$
In particular, $\tau_k,$ $k\in \N$, depends on the realization of all previously chosen perturbation parameters $\alpha_1, \alpha_2, \hdots, \alpha_{k}$. This makes the expression in Proposition \ref{p-recursiveac} too cumbersome to work with in many cases.

%%%%%%%%%%%%%%%%%%%%%%%%%%%%%%%
%%%%%%%%%%%%%%%%%%%%%%%%%%%%%%%
\subsection{Rademacher Potential}\label{ss-rademacher}
%%%%%%%%%%%%%%%%%%%%%%%%%%%%%%%
%%%%%%%%%%%%%%%%%%%%%%%%%%%%%%%

\indent The concepts developed in the previous two Subsections can be applied with different choices of the perturbation parameters. We wish to determine whether certain iterated operators, or classes of them, localize ($\lim_{k\to\infty}\tau_k=0$) or delocalize ($\lim_{k\to\infty}\tau_k>0$) and what conditions are necessary and/or sufficient for such behavior. It is now clear that much of the previous construction can be easily adapted to various scenarios, so we make a slightly different choice of our starting measure. Let the starting measure be chosen as $\widetilde{\mu}_0(x)=\frac{1}{2c}\chi\ci{[-c,c]}(x)dx$. Recall that the choice of a constant function here is possible as long as the beginning weight function is in $L^1\ti{loc}(-c,c)$. This is because a unitary operator can then be applied, as described in Section \ref{s-startiterate}, to begin with a constant weight function. Hence, the interval of support is more desirable to generalize.

The simplest scenario is to start with the perturbation parameters given by $\{\al_n\}_{n=1}^k$ chosen with respect to a Rademacher distribution, i.e.~$\al_n=\pm c$. These parameters collectively take the place of the potential in the description of Anderson-type Hamiltonians, and in particular the discrete Schr\"odinger operator, described in Section \ref{ss-01}. Consequently, we refer to the choice of $\al_n=\pm c$, $n=1,\dots,k$, as defining a Rademacher potential.

\begin{theo}
\label{t-Rademacher}
The operator constructed in the previous three sections, when the $\{\al_k\}$'s are chosen i.i.d.~with respect to the probability measure $\Pro=\frac{1}{2}\delta\ci{-c}+\frac{1}{2}\delta\ci{c}$ (Rademacher potential), localizes for any fixed disorder $c>0$.
\end{theo}

\begin{proof}
Proposition \ref{p-recursiveac} with $\al_n=\pm c$ for all $n$ reads:
\begin{align*}
\|(\widetilde{\mu}_{\al_k})\ci{ac}\|=1-\dfrac{1}{c^2}\sum_{n=1}^{k}\dfrac{ e^{1/c\tau_{n-1}}}{\tau_{n-1}(e^{1/c\tau_{n-1}}-1)^2}.
\end{align*}
\indent We are mainly concerned with the exact limiting value of this series. The convergence of this series is clear: Physics tells us that the absolutely continuous part of $\widetilde{\mu}_{\al_k}$ cannot become negative, it is bounded above by 1, and the sequence of partial sums decreases as $k$ increases.

The summand can be rearranged by expanding the denominator and factoring out a term of $e^{1/c\tau_{n-1}}$ to yield
\begin{align*}
\dfrac{ e^{1/c\tau_{n-1}}}{\tau_{n-1}(e^{1/c\tau_{n-1}}-1)^2} = \dfrac{1}{\tau_{n-1}(e^{1/c\tau_{n-1}}-2+e^{-1/c\tau_{n-1}})}.
\end{align*}
Hence, localization necessitates
\begin{align*}
\lim_{n\to\infty}[\tau_{n-1}(e^{1/c\tau_{n-1}}-2+e^{-1/c\tau_{n-1}})]=\infty.
\end{align*}
In this specific scenario, the operator began with a total mass of 1. This implies $0\leq \tau_{n-1}\leq 1$. Hence, for fixed $c$, we have:
\begin{align*}
\lim_{n\to\infty}[\tau_{n-1}(e^{1/c\tau_{n-1}}-2+e^{-/c\tau_{n-1}})]=\infty&\iff\lim_{n\to\infty}e^{1/c\tau_{n-1}}=\infty \\
&\iff\lim_{n\to\infty}\tau_{n-1}=0
\end{align*}
The first if and only if statement can be verified by noticing that the exponential is ``stronger'' than the $\tau_{n-1}$ term. Also, $e^{-1/c\tau_{n-1}}$ remains bounded for $0\leq \tau_{n-1}\leq 1$ by $e^{-1/c}$. Therefore, we conclude that if $\tau_{k-1}\not\to 0$ as $k\to\infty$, then the sum does not converge. This is a contradiction, so it must be that $\tau_{k-1}\to 0$ as $k\to\infty$, and the operator localizes.
\end{proof}

%%%%%%%%%%%%%%%%%%%%%%%%%%%%%%%
%%%%%%%%%%%%%%%%%%%%%%%%%%%%%%%
\section{Bounds for Non-Orthogonal Perturbations}\label{s-northo}
%%%%%%%%%%%%%%%%%%%%%%%%%%%%%%%
%%%%%%%%%%%%%%%%%%%%%%%%%%%%%%%

In the previous sections, the recursive nature of our construction was necessary in order to ensure that each perturbation was orthogonal to previous directions and so that we were able to explicitly carry out successive computations. This guaranteed that the absolutely continuous spectrum was not increased due to  via reabsorption from the point masses outside of the interval $[-1,1]$.

In an effort to escape these restrictions, we now consider the case where a perturbation in a general direction $\f$ is applied to an operator. Only its overlap with the the point masses of the unperturbed operator should be known. The results are contributions to abstract rank-one perturbation theory. From this perspective, the estimates provide fundamental bounds on the effects of a single perturbation. They are also useful for creating examples, as equation \eqref{e-bound} identifies which factors influence the shifting of mass from one type of spectrum to another. 

Such perturbations can be interpreted as representative of one that would arise from the constructive iteration scheme after $N$ steps, in which the new direction vector $\f_k$ would then not be assumed orthogonal to the previous perturbation vectors. Here, the acquired bounds are not sufficiently sharp to allow for the choice of the perturbation parameters with respect to probability measures other than Rademacher.

Theorems \ref{t-worstcase} and \ref{t-singular} show the maximum/minimum amounts of absolutely continuous and pure point spectrum that can be created/destroyed by such a perturbation. The method of proof requires only knowledge of Aronszajn--Donoghue theory (Theorem \ref{t-AD}) and the integral transforms therein. In particular, the theorems represent a worst-case scenario in each situation, similar to using Rademacher potentials in Section \ref{ss-rademacher}. Surprisingly, sufficient conditions to gain pure point or lose absolutely continuous spectrum are also achieved, see Proposition \ref{p-worstcase} and the comments after the theorems.

\begin{theo}\label{t-worstcase}
Let $\mu\in M_{+}(\RR)$ be such that 
$$d\mu(x)=f(x)\chi\ti{[-a,a]}(x)\text{dx}+d\mu\ti{s}(x)
\quad\text{where}\quad \mu\ti{s}=\sum_{n=1}^N \al_n \delta_{x_n},$$ 
as well as $f\in L^2[-a,a]$, $\|\mu\|=1$ and $\sum_{n=1}^N\al_n=c$. Let $\f\in L^2(\mu)$ be a unit vector with $$\sum_{n=1}^N\al_n|\f (x_n)|^2<\varepsilon.$$ 
Let the spectral measure of the self-adjoint operator
$M_t+\lambda\langle\fdot,\f\rangle\ci{{L^2(\mu)}}\f$ on $L^2(\mu),$
with respect to $\f$, be denoted by $\mu_{\lambda}$. Assume that I is a compact interval not including 0. Then for all $\lambda\in$I, there exists real $k>0$ such that the spectral measure $\mu_{\lambda}$ satisfies 
$$\|(\mu_{\lambda})\ti{ac}\|\leq\|\mu\ti{ac}\|-k.$$
\end{theo}

The theorem states that there is a minimum amount of absolutely continuous spectrum lost after a general perturbation. Note also that $\varepsilon\leq 1$ is required due to the assumption $\|\mu\|=1$. More explicit estimates for $k$, and its dependence on $\varepsilon$ and $\lambda$, are given in Proposition \ref{p-worstcase}.

\begin{proof}
In order to simplify notation, all inner products are taken in $L^2(\mu)$ unless otherwise stated. Assume the hypotheses on $f$, $\mu$ and $\f$ above. Decompose both $\f$ and $\mu$ into two parts, one concerning the absolutely continuous spectrum on the interval $[-a,a]$, and the other concerning the $N$ point masses. Hence, we define 
\begin{align*}
\widetilde{\f}=\f \chi\ti{[-a,a]}, \ \ \f_p=\f-\widetilde{\f}, \ \ d\mu_{ac}(x)=f\chi\ti{[-a,a]}(x)dx, \ \text{and} \ \mu_p=\sum_{n=1}^N\al_n \delta_{x_n}.
\end{align*}
The rank-one perturbation $\lambda\langle\fdot,\f\rangle\f$ can now be broken down in terms of $\widetilde{\f}$ and $\f_p$ so that the interaction between each part of the perturbation and the absolutely continuous spectrum can be estimated. We begin by estimating the norm of the perturbation:
\begin{align}
\notag
\|\lambda\langle\fdot,\f\rangle\f\|
&=
|\lambda|
\,
\|\langle\fdot,(\widetilde{\f}+\f_p) \rangle(\widetilde{\f}+\f_p)\| \\
\label{e-bound}
&\leq |\lambda|
\,
\left(\|\langle\fdot,\widetilde{\f}\rangle\widetilde{\f}\|+\|\langle\fdot, \widetilde{\f}\rangle\f_p\|+\|\langle\fdot,\f_p\rangle\widetilde{\f}\|+\|\langle\fdot,\f_p\rangle\f_p\|\right)
\end{align}

The four terms in the inequality \eqref{e-bound} will be discussed and evaluated separately. The first term, $|\lambda|\|\langle\fdot,\widetilde{\f}\rangle\widetilde{\f}\|$, involves only $\widetilde{\f}$  so the perturbation by this factor has no relevance to the point masses and is in fact orthogonal to $\mu_p$. 

This term recreates a setting from our earlier results, where the perturbation vector was not concerned with the previous point masses due to orthogonality. The perturbation therefore has the effect of creating a single eigenvalue in the new spectral measure, determined solely by $\lambda$. The spectral measure was then a constant on an interval thanks to our use of an auxiliary space and a choice of $\f$, so the mass of this eigenvalue was easy to compute. We have no such luxury here, as the choice of $f$ has only the restriction that $f\in L^2[-a,a]$. Therefore, we solace ourselves with the ability to prove that there is a minimum for the mass of this eigenvalue, when $\lambda$ is chosen from a compact interval not containing 0. Estimates for this value are attained in Proposition \ref{p-worstcase} below, with the corresponding loss of sharpness to the global estimate of $k$.

Let $h:\RR\to\RR$ be defined by $h(\lambda)=\mu_{\lambda}\{x_{\lambda}\}$, the mass of the created eigenvalue $x_{\lambda}$ in the spectral representation $\mu_{\lambda}$. The explicit calculation of
the location $x_\lambda\in \R\backslash [-a,a]$ and the strength $\mu_{\lambda}\{x_{\lambda}\}$ of the new eigenvalue is given by Aronszajn--Donoghue Theory. The process defining the function $h(\lambda)$ is thus given by two integrals:
\begin{align}\label{e-FG}
\int_{-a}^a\dfrac{f(t)dt}{t-x_\lambda}=-\dfrac{1}{\lambda} \hspace{7mm}\text{and}\hspace{7mm} \mu_{\lambda}\{x_{\lambda}\}=G(x_{\lambda})=\int_{-a}^a\dfrac{f(t)dt}{(t-x_{\lambda})^2}.
\end{align}

Both integrals are finite and yield $C^1$ functions because $f\in L^2[-a,a]$. We conclude that $h(\lambda)$ is itself continuous, as the composition of continuous functions. Note that if $\widetilde{\f}$ were acting on $\mu$, not just $\mu\ti{ac}$, then $h(\lambda)$ is not necessarily continuous as the point masses interfere with the integral. If $\lambda$ is chosen from a compact interval $I\subset\RR\setminus\{0\}$ then $h(\lambda)$ must achieve a minimum value on $I$. For the remainder of the paper, we use the definition
\begin{align*}
   d:=\min_{\lambda\in I}h(\lambda).
\end{align*}

The second and fourth terms on the right hand side of the inequality \eqref{e-bound} are not relevant. Indeed, those two factors deal only with changes to the pure point spectrum, as the perturbations are in the ``direction" of $\f_p$. Hence, the individual perturbations do not cause any change to the unperturbed absolutely continuous spectrum due to orthogonality of $\mu_p$ and $\mu\ti{ac}$.

The third term on the right hand side of the inequality \eqref{e-bound}, $|\lambda|\|\langle\fdot,\f_p\rangle\widetilde{\f}\|\leq|\lambda|\|\f_p\|\|\widetilde{\f}\|$, can be handled with our assumptions and above calculations. Indeed, $\f_p$ only interacts with $\mu_p$ by definition, so
\begin{align*}
\|\f_p\|^2&=\langle\f_p,\f_p\rangle\ci{L^2(\mu_p)}=
\sum_{n=1}^N\al_n|\f_p(x_n)|^2=\sum_{n=1}^N\al_n |\f(x_n)|^2\leq\varepsilon.
\end{align*}
Similar reasoning yields that $\|\widetilde{\f}\|=\sqrt{1-\varepsilon}$. The estimate for this term is then $\|\langle\fdot,\f_p\rangle\widetilde{\f}\|\leq \sqrt{\varepsilon}\sqrt{1-\varepsilon}$.

Overall, we observe that the first term is how much the absolutely continuous spectrum is decreased by the creation of the new eigenvalue. The third term is correcting for what happens to the point masses, as there is no guarantee that some of the mass in $\mu_p$ doesn't reenter the interval $[-a,a]$ due to the effect of $\f$. Moreover, the intertwining operator $V_{\lambda}$ for the spectral theorem is unitary so the essential spectrum remains unchanged under our compact perturbation and there is no total mass lost. We can now conclude 
\begin{align*}
\|(\mu_{\lambda})\ti{ac}\|\leq\|\mu\ti{ac}\|-\left[d-|\lambda|\sqrt{\varepsilon}\sqrt{1-\varepsilon})\right].
\end{align*}
The theorem follows.
\end{proof}

In general, we cannot assume that $\|(\mu_\lambda)\ti{ac}\|\le\|\mu\ti{ac}\|.$ Therefore, it is imperative that $d-|\lambda|\sqrt{\varepsilon}\sqrt{1-\varepsilon}>0$ for the previous result to not be vacuous. Let $|\lambda|\ti{max}$ denote the maximum value of $|\lambda|$ on $I$. The desired inequality is achieved when 
\begin{align*}
d>|\lambda|\ti{max}\sqrt{\varepsilon-\varepsilon^2}.
\end{align*}
It is noteworthy that $d$ was constructed to depend upon both $\lambda$ and the a.c.~spectral mass, which directly relates to $c$ and $\varepsilon$. The value of $d$ can thus be estimated by these constants.

\begin{prop}\label{p-worstcase}
Let $\lambda$ be chosen from a compact interval $I\subset \R$ not including 0. Then 
\begin{align*}
d\geq\dfrac{1-c}{(a+|\lambda|\ti{max}(1-\varepsilon)+1)^2},
\end{align*}
where $|\lambda|\ti{max} = \max_{\lambda\in I}|\lambda|$, and $d$, $c$ and $\varepsilon$ are as in the proof of Theorem \ref{t-worstcase}.
\end{prop}
\begin{proof}
Without loss of generality, we can assume that $f$ is positive on the interval $[-a,a]$ and that $\lambda>0$. This means that the eigenvalue created by the $\lambda$ perturbation by $\widetilde{\f}$ will be to the right of the interval, i.e.~$|x_{\lambda}|>a$. Also, take $\lambda=|\lambda|\ti{max}$. Recall the formulas in equation \eqref{e-FG}. In order to minimize $\mu_{\lambda}\{x_{\lambda}\}$, we minimize the kernel of the integral operator $G(x)$. This minimum occurs when $f$ is represented by a delta mass at the endpoint $\{-a\}$ so that the eigenvalue will fall as close to $a$ as possible. This delta mass is of strength $1-c$ by necessity and the integration $G(x_{\lambda})$ can be computed to find
\begin{align*}
\mu_{\lambda}\{x_{\lambda}\}\geq\dfrac{1-c}{(x_{\lambda}+1)^2}.
\end{align*}
To minimize this inequality, we must maximize the value of $x_{\lambda}$. The distance $x_{\lambda}$ is placed from the endpoint ${a}$ must be less than 
\begin{align*}
\|\lambda\langle\fdot,\widetilde{\f}\rangle\widetilde{\f}\|=\lambda(1-\varepsilon).
\end{align*}
This means that $x_{\lambda}\leq a+\lambda(1-\varepsilon)$, and we can conclude that 
\begin{align*}
\mu_{\lambda}\{x_{\lambda}\}\geq\dfrac{1-c}{(x_{\lambda}+1)^2}\geq\dfrac{1-c}{(a+\lambda(1-\varepsilon)+1)^2},
\end{align*}
as desired.
\end{proof}

This approximation of $d$ can be applied to the case where $f$ is a constant, which occurred at each step of the iterative construction.

\begin{cor}\label{t-rademacherexample}
Let $\mu\in M_{+}(\RR)$ be such that 
$$d\mu(t)=f(t)\chi\ti{[-a,a]}(t)\text{dt}+d\mu\ti{s}(t)
\quad\text{where}\quad \mu\ti{s}=\sum_{n=1}^N \al_n \delta_{x_n},$$ 
where we define  $f(t)=w_N(t)$ such that $f\in L^1\ti{loc}$, $\|\mu\|=1$ and $\sum_{n=1}^N\al_n=c$. Furthermore, let $\f\in L^2(\mu)$ such that $\f|\ti{[-a,a]}(t)=1/\sqrt{2w_N(t)}$, $\|\f\|=1$ and $\sum_{n=1}^N\al_n|\f (x_n)|^2<\varepsilon$. Assume that I is a compact interval not including 0. Then for all $\lambda\in$I, we have the following inequality  
$$\|(\mu_{\lambda})\ti{ac}\|\leq\|\mu\ti{ac}\|-\left[\dfrac{e^{1/\lambda\tau_N}}{\lambda^2\tau_N(e^{1/\lambda\tau_N}-1)^2}-\dfrac{e^{1/\lambda\tau_N}\sqrt{\varepsilon}}{\lambda\tau_N(e^{1/\lambda\tau_N}-1)^2}\right].$$
\end{cor}
\begin{proof}
See the proof of the previous Theorem. In this case we have the assumption that $$\widetilde{\f}(t)=\dfrac{1}{\sqrt{2w_N(t)}}.$$ The notation $\widetilde{\f}$ should not be confused with the image of $\f$ under a unitary operator as in previous Sections. However, recall that $w_N(t)$ is simply representing a weight function and matches the notation developed in Section \ref{s-infinite}. When $\lambda$ is chosen, i.e.~when Rademacher potentials are used, it is then possible to explicitly calculate the value of $d$. If a choice of $\lambda$ is not imposed, simply pick $\lambda$ in the formula to be $|\lambda|\ti{max}$ to obtain a general bound.
\end{proof}

Similarly, we deduce how the singular part is effected by the perturbation at a single step.

\begin{theo}\label{t-singular}
Let $\mu\in M_{+}(\RR)$ be such that 
$$d\mu(t)=f(t)\chi\ti{[-a,a]}(t)\text{dt}+d\mu\ti{s}(t)
\quad\text{where}\quad \mu\ti{s}=\sum_{n=1}^N \al_n \delta_{x_n},$$ 
where $f\in L^2(m)$, $\|\mu\|=1$ and $\sum_{n=1}^N\al_n=c$. Furthermore, let $\f\in L^2(\mu)$, $\|\f\|=1$ and $$\sum_{n=1}^N\al_n|\f (x_n)|^2<\varepsilon.$$ 
Let the spectral measure of the self-adjoint operator
$$M_t+\lambda\langle\fdot,\f\rangle\ci{{L^2(\mu)}}\f \quad\text{on}\quad L^2(\mu),$$
with respect to $\f$, be denoted by $\mu_{\lambda}$. Assume that I is a compact interval not including 0. Then for all $\lambda\in$I, there exists $k\in\RR$ such that the spectral measure $\mu_{\lambda}$ satisfies 
$$\|(\mu_{\lambda})\ti{s}\|\geq\|\mu\ti{s}\|+k.$$
\end{theo}

\begin{proof}
We employ a similar strategy to the one used in Theorem \ref{t-worstcase}. Namely, decompose $\f$ into $\widetilde{\f}$ and $\f_p$ and estimate the $\lambda$ perturbation:
\begin{align*}
\notag
\|\lambda\langle\fdot,\f\rangle\f\|
&=
|\lambda|
\,
\|\langle\fdot,(\widetilde{\f}+\f_p) \rangle(\widetilde{\f}+\f_p)\| \\
&\leq |\lambda|
\,
\left(\|\langle\fdot,\widetilde{\f}\rangle\widetilde{\f}\|+\|\langle\fdot, \widetilde{\f}\rangle\f_p\|+\|\langle\fdot,\f_p\rangle\widetilde{\f}\|+\|\langle\fdot,\f_p\rangle\f_p\|\right).
\end{align*}
We are only concerned with the first, second and fourth terms in the inequality, as they affect $\f_p$. The first term is responsible for creating an eigenvalue of strength at least $d$, as estimated above. The fourth term  has no effect, as the essential spectrum of an operator does not change under a rank-one perturbation. This means that the eigenvalues are shifted and masses are redistributed according to this term, but their total mass is the same because it cannot other kinds of spectrum. Estimating the second term is analogous to the mixed term in Theorem \ref{t-worstcase} and yields an effect of $|\lambda|\sqrt{\varepsilon}\sqrt{1-\varepsilon}$. Hence, the singular mass increases by a created eigenvalue and is adjusted for possible mass entering the absolutely continuous spectrum by a mixed term. Our conclusion thus follows its absolutely continuous counterpart and we set $k=d-|\lambda|\sqrt{\varepsilon}\sqrt{1-\varepsilon}$ to yield the Theorem.
\end{proof}

The same restrictions are relevant to applications of this theorem as to Theorem \ref{t-worstcase}. In general, we cannot assume that $\|(\mu_{\lambda})_s\|\geq\|\mu_s\|$. For the result to not be vacuous, we must ensure that $d-|\lambda|\sqrt{\varepsilon}\sqrt{1-\varepsilon}>0$. Hence, it is required that
\begin{align*}
d>|\lambda|\ti{max}\sqrt{\varepsilon-\varepsilon^2}.
\end{align*}
The symmetry of Theorems \ref{t-worstcase} and \ref{t-singular} adds further validation to the estimates.

%%%%%%%%%%%%%%%%%%%%%%%%%%%%%%%
%%%%%%%%%%%%%%%%%%%%%%%%%%%%%%%
\section{Appendix: Choosing Orthogonal Direction Vectors} \label{App:AppendixA}
%%%%%%%%%%%%%%%%%%%%%%%%%%%%%%%
%%%%%%%%%%%%%%%%%%%%%%%%%%%%%%%

This elementary proof is included for the convenience of the reader, and is motivated by Theorem 2.10 in \cite{Folland} and the definition of the Lebesgue integral.

\begin{lem}
\label{l-APP}
Let $S=\{f_n\}_{n=1}^N$ be a finite set of functions orthogonal in a separable Hilbert space $L^2(\eta)$, where $\eta$ is a positive Borel measure supported on $[-1,1]$ without a point mass at $x=1$. Then there exists a measurable function $h(x)$ with $|h(x)|=1$ a.e.~with respect to $\eta$, so that the set $S\cup \{h\}$ is orthogonal.
\end{lem}

\begin{proof}
Without loss of generality, we consider the positive parts of each $f_n$, written as
$f_n^+(x):=\max\{f_n(x), 0\}$. Let $\{g_m^n\}_{m\in\NN}$ be the sequence of simple functions in standard representation which approximates $f_n^+$ pointwise and uniformly (wherever $f_n^+$ is bounded). \\ 
\indent Let $E_m^n$ denote the partition of $[-1,1)$ on which $g_m^n$ is constant. For $n=1, \dots, N$, take the union of the endpoints of $E_m^n$ and cover $[-1,1)$ by non-overlapping half-open intervals corresponding to this union. Denote the collection of these intervals by $I_m$. Then, for each fixed $m$, $g_m^n$, $n=1, \dots, N$, is constant on each half-open interval $I\subset I_m$.\\ 
\indent For each $I\subset I_m$ define 
\[ 
h_m|\ci{I}=
	\begin{cases} 
      0 & \text{on } [-1,1)\setminus I \\
      1 & \text{on the right half of } I \\
      -1 & \text{on the left half of } I
   \end{cases}
\]
and $h_m:=\sum_I h_m|_I$. This gives us that $\left<g_m^n,h_m\right>=0$, $\forall m,n$, and $h_m$ converges with respect to $\eta$ to some measurable $h$ with $|h(x)|=1$ on $[-1,1)$.\\ 
\indent All that remains to show is that $\left<f_n,h\right>=0$, $\forall n$. This follows by a simple application of the Dominated Convergence Theorem to the functions $g_m^n(x)$ and $h(x)$:
\begin{align*}
\left<f_n,h\right> =\int_{-1}^{1}\lim_{m\to\infty}g_m^n(x)h_m(x)d\eta(x) 
=\lim_{m\to\infty}\int_{-1}^{1}g_m^n(x)h_m(x)d\eta(x) 
=0 
\end{align*}
for all $n$.
\end{proof}

%%%%%%%%%%%%%%%%%%%%%%%%%%%%%%%
%%%%%%%%%%%%%%%%%%%%%%%%%%%%%%%

\end{document}